\documentclass[reqno]{amsart}%
\usepackage{amsmath,amsfonts,amsthm,amssymb,color}
\usepackage{geometry}
\usepackage{hyperref}
\usepackage{graphicx}
\usepackage{pdfsync}
\usepackage{mhequ}
\usepackage{graphicx}
\usepackage{amsmath}
\usepackage{amsfonts}
\usepackage{amssymb}%
\providecommand{\U}[1]{\protect\rule{.1in}{.1in}}

\newtheorem{theorem}{Theorem}[section]
\theoremstyle{plain}

\newtheorem{condition}{Condition}

\newtheorem{corollary}[theorem]{Corollary}

\newtheorem{definition}[theorem]{Definition}

\newtheorem{lemma}[theorem]{Lemma}
\newtheorem{notation}{Notation}

\newtheorem{proposition}[theorem]{Proposition}
\newtheorem{remark}[theorem]{Remark}

\numberwithin{equation}{section}

\begin{document}
\title[Evolving Communities]{ Evolving Communities with Individual Preferences}
\author{Thomas Cass}
\address{Thomas Cass, Department of Mathematics, Imperial College London, The Huxley
Building, 180 Queensgate, London.}
\email{thomas.cass@imperial.ac.uk}
\author{Terry Lyons}
\address{Oxford Man Institute }
\thanks{The work of both authors was supported by EPSRC grant EP/F029578/1. The
research of Terry Lyons is supported by EPSRC grant EP/H000100/1 and the
European Research Council under the European Union's Seventh Framework
Programme (FP7-IDEAS-ERC) / ERC grant agreement nr. 291244. Terry Lyons
acknowleges the support of the Oxford-Man Institute.}
\maketitle

\begin{abstract}
The goal of this paper is to provide mathematically rigorous tools for
modelling the evolution of a community of interacting individuals. We model
the population by a measure space $\left(  \Omega,\mathcal{F},\nu\right)  $
where $\nu$ determines the abundance of individual preferences. The
preferences of an individual $\omega\in\Omega$ are described by a measurable
choice $X\left(  \omega\right)  $ of a rough path.

We aim to identify, for each individual, a choice for the forward evolution
$Y_{t}\left(  \omega\right)  $ for an individual in the community. These
choices $Y_{t}\left(  \omega\right)  $ must be consistent so that
$Y_{t}\left(  \omega\right)  $ correctly accounts for the individual's
preference and correctly models their interaction with the aggregate behaviour
of the community.

We focus on the case of weakly interacting systems, where we are able to
exhibit the existence and uniqueness of consistent solutions.

In general, solutions are continuum of interacting threads analogous to the
huge number of individual atomic trajectories that together make up the motion
of a fluid. The evolution of the population need not be governed by any
over-arching PDE. Although one can match the standard nonlinear parabolic PDEs
of McKean-Vlasov type with specific examples of communities in this case. The
bulk behaviour of the evolving population provides a solution to the PDE.

An important technical result is continuity of the behaviour of the system
with respect to changes in the measure $\nu$ assigning weight to individuals.
Replacing the deterministic $\nu$ with the empirical distribution of an i.i.d.
sample from $\nu$ leads to many standard models, and applying the continuity
result allows easy proofs for propagation of chaos.

The rigorous underpinning presented here leads to uncomplicated models which
have wide applicability in both the physical and social sciences. We make no
presumption that the macroscopic dynamics are modelled by a PDE.

This work builds on the fine probability literature considering the limit
behaviour for systems where a large no of particles are interacting with
independent preferences; there is also work on continuum models with
preferences described by a semi-martingale measure. We mention some of the key papers.

\end{abstract}

\section{Introduction}

Consider a community $\Omega=\left\{  \omega_{i},i\in1,\ldots,N\right\}  $ of
$N$ individuals who at time zero are allocated positions $\left(  Y_{0}\left(
\omega\right)  \right)  _{\omega\in\Omega}$ in some state space, and suppose
that these individuals have preferences which determine how they evolve in
their environment. Let the evolution of individual $\omega_{i}$ be denoted
$Y^{i},$ and suppose $X^{i}$ is the preference of this individual. To model
this situation we assume that $Y^{i}$ and $X^{i}$ are related via
\[
dY^{i}=f\left(  Y^{i}\right)  dX^{i}.
\]
This equation has a unique meaning when $X^{i}$ is a $p-$rough path and $f$ is
Lip-$\gamma$, $\ \gamma>p$ (see \cite{LQ02}). By imposing a measure on
$\Omega$ (e.g. the counting measure), we can introduce $\Gamma=$Law$\left(
Y\right)  $, and then consider the situation where the evolution $Y^{i}$ is
influenced by the wider population through $\Gamma.$ This leads to equations
of the form%
\[
dY=f\left(  Y\right)  dX+g\left(  Y,\Gamma\right)  dt.
\]

Deterministic models of this type are commonplace in the modelling of physical
systems. For example, our individuals might be planets. Although planets do
not have preferences, they are subject to gravitational forces exerted by
other planets which affect the state $Y=\left(  p,q\right)  $ of their
position and momentum. In this setting, we can model the Newtonian evolution
of the locations of the individuals via a differential equation:%

\begin{align*}
dp_{t}\left(  \omega\right)   &  =-\sum_{\tilde{\omega}\in\Omega
\backslash\left\{  \omega\right\}  }Gm\left(  \omega\right)  m\left(
\tilde{\omega}\right)  \frac{q_{t}\left(  \omega\right)  -q_{t}\left(
\tilde{\omega}\right)  }{\left\vert q_{t}\left(  \omega\right)  -q_{t}\left(
\tilde{\omega}\right)  \right\vert ^{3}}dt\\
dq_{t}\left(  \omega\right)   &  =\frac{1}{m\left(  \omega\right)  }%
p_{t}\left(  \omega\right)  dt\\
G  &  =6.6730010^{-11}\text{N(m/kg})^{2}.
\end{align*}
Unless there is a collision, the theory of ordinary differential equations is
an adequate tool to describe the short time evolution of this system. If each
member of the community carried charge of the same sign then the equations
would change again
\begin{gather*}
dp_{t}\left(  \omega\right)  =\sum_{\tilde{\omega}\in\Omega\backslash\left\{
\omega\right\}  }\frac{_{\mu_{0}}}{4\pi}\frac{1}{m\left(  \omega\right)
m\left(  \tilde{\omega}\right)  }C\left(  \omega\right)  C\left(
\tilde{\omega}\right)  \frac{q_{t}\left(  \omega\right)  -q_{t}\left(
\tilde{\omega}\right)  }{\left\vert q_{t}\left(  \omega\right)  -q_{t}\left(
\tilde{\omega}\right)  \right\vert ^{3}}p_{t}\left(  \omega\right)  \times
dq_{t}\left(  \tilde{\omega}\right) \\
dq_{t}\left(  \omega\right)  =\frac{1}{m\left(  \omega\right)  }p_{t}\left(
\omega\right)  dt\\
C\text{ is the charge on }\omega\text{ and }\mu_{0}\text{=}4\pi10^{-7}%
\text{N/A is the magnetic permeability.}%
\end{gather*}
These examples are very different in detail because of the involvement of the
$dq_{t}.$ But both equations capture systems of physical interest, where it is
natural to consider the evolution of the population as a whole, and to
understand what happens when particles are replaced by more particles with
proportionately smaller mass (or charge) in the continuum limit. In this case
one would hope and expect that the particle and/or current density would solve
the appropriate Vlasov style equation.

Of course, there are a huge number of similar if less precisely characterised
models: in the social sciences, in the modelling of the evolution of cancer,
etc. where the evolution of an individual is affected by the evolution of the
wider community. It may not always be the case that the interaction with the
population is pairwise and more generally, one might expect to consider
equations of the form
\[
dY_{t}\left(  \omega\right)  =\phi\left(  Y_{t}\left(  \omega\right)
,\mathbb{\nu}_{\ast}Y_{t}\right)  dt,
\]
where $\mathbb{\nu}_{\ast}Y_{t}$ is the push forward of $\mathbb{\nu}$ giving
the mass distribution of $Y$ at time $t,$ and $Y_{t}\left(  \omega\right)  $
is a solution to the above equation for $\mu$-every $\omega.$

Individual differences mean that different individuals will respond
differently to the same external environment. We can easily make adaptations
to the calculus to take such behaviour into account. One is lead to equations
of the following kind:%
\[
dY_{t}\left(  \omega\right)  =\phi\left(  Y_{t}\left(  \omega\right)
,\mathbb{\nu}_{\ast}Y_{t}\right)  dt+\xi\left(  Y_{t}\left(  \omega\right)
,\mathbb{\nu}_{\ast}Y_{t}\right)  dX_{t}\left(  \omega\right)
\]
where $X\left(  \omega\right)  $ represents the individual preferences of the
individual $\omega.$ Now if $X$ is smooth there is no additional difficulty.
If $\nu$ is a probability measure, and $X$ is a semi-martingale under this
measure then (under regularity conditions) Sznitman \cite{sznitman}, Kurtz
\cite{kurtz1} , M\'{e}l\'{e}ard \cite{mel} and others have proved that the
corresponding particle system obtained by taking an i.i.d. sample from $X$ and
using this empirical measure in the above equations provides a converging
sequence of particle systems. The limit can be identified with the law of a
non-linear PDE which solves the Vlasov equation. Dawson and G\"{a}rtner
\cite{dawson} have important results on the large deviations in the
convergence of the weakly interactive system (where $\xi\left(  Y_{t}\left(
\omega\right)  ,\mathbb{\nu}_{\ast}Y_{t}\right)  $=$\xi\left(  Y_{t}\left(
\omega\right)  \right)  $ and later den Hollander \cite{hol} and Guionnet
\cite{Gui3}, \cite{Gui1},\cite{Gui2} and \cite{Gui4} considered the large
deviations for interaction in a random media in problems arising from the
dynamics of spin glasses. Kurtz promoted more advanced discussion in
\cite{kurtz2}.

In many cases of interest, it is unreasonable to expect the preferences $X$ to
be a semi-martingale as evidenced by the sucess of fractional Brownian motion
in the modelling of fluids (see \cite{hairer} and the references therein). In
addition, individuals often have knowledge that makes the previsible
assumption equally inappropriate. We now understand that the natural
assumption on $X$ that leads to equation with a strong meaning is that $X$
should be rough path. Indeed there are a large number of deterministic (and
numerically approximable) systems that evolve without the assistance of a PDE.
We will study a mathematical framework which exposes the consequences of
persistent differences between individuals in the population dynamics (see
\cite{coulson} for a study of such a phenomenon in the context of red deer populations.)

The McKean-Vlasov model leads one, in the limit, to the equation
\[
dY_{t}\left(  \omega\right)  =\phi\left(  Y_{t}\left(  \omega\right)
,\mathbb{\nu}_{\ast}Y_{t}\right)  dt+\xi\left(  Y_{t}\left(  \omega\right)
,\mathbb{\nu}_{\ast}Y_{t}\right)  dW_{t}\left(  \omega\right)
\]
where the individual preferences are given by a $d$-dimensional Wiener measure
$W$. However individuals can have very different volatility and speed of
reaction to events. Let $\sigma$ be a positive real function on the space
$\Omega$ of individuals and consider the equation%
\[
dY_{t}\left(  \omega\right)  =\phi\left(  Y_{t}\left(  \omega\right)
,\mathbb{\nu}_{\ast}Y_{t}\right)  dt+\xi\left(  Y_{t}\left(  \omega\right)
,\mathbb{\nu}_{\ast}Y_{t}\right)  \sigma\left(  \omega\right)  dW_{t}\left(
\omega\right)  .
\]
For appropriate $\phi,\xi$ and paths $\mu_{s}$ in measures on $Y$-space one
can consider the indexed family of differential equations, one for each
$\omega,$%
\begin{align*}
dY_{t}\left(  \omega\right)   &  =\phi\left(  Y_{t}\left(  \omega\right)
,\mu_{t}\right)  dt+\xi\left(  Y_{t}\left(  \omega\right)  ,\mu_{t}\right)
\sigma\left(  \omega\right)  dW_{t}\left(  \omega\right)  .\\
&  Y_{0}\left(  \omega\right)  \text{ given}%
\end{align*}
For almost every $\omega$ the path $t\rightarrow\sigma\left(  \omega\right)
W_{t}\left(  \omega\right)  $ is a geometric rough path of finite p-variation
for every $p>2$ based on rescaling $W$ and its Levy area. If $\nu_{t}$ is a
path of finite variation in the space of measures and $\phi,\xi$ are at least
$C^{2+\varepsilon}$ then it will be the case that the rough path solution
$Y_{t}\left(  \omega\right)  $ to this equation will exist and be unique.
Considering all $\omega,$we see that $Y_{t}\left(  \omega\right)  $ is a
random variable and we denote its law by the probability measure $\tilde{\mu
}_{t}.$ Of course, this new path $t\rightarrow\tilde{\mu}_{t}$ in measures
will not in general coincide with the path $t\rightarrow\mu_{t}.$ But it makes
complete sense to ask whether there is a choice $t\rightarrow\mu_{t}$ so that
the resultant measure path $t\rightarrow\tilde{\mu}_{t}$ does coincide with
it. In this case we have a community of individuals evolving according to
their individual preferences in a way that is also consistent with the
dynamics of the population as a whole.

We note that in general having individuals with different volatility results
in a process $t\rightarrow\sigma\left(  \omega\right)  W_{t}\left(
\omega\right)  $ that is far from a semimartingale against the Wiener measure
and using the base filtration. One cannot have $\sigma\left(  \omega\right)  $
measurable in $\mathcal{F}_{0}$ unless one enlarges the filtration or $\sigma$
is constant. The lack of previsibility does not impede the rough path
perspective, and there is no issue at all in setting up the equations. One
theoretically amusing choice for $\sigma$ is to take%
\[
\sigma\left(  \omega\right)  =\frac{1}{\sup_{t\in\left[  0,1\right]
}\left\vert W_{t}\left(  \omega\right)  \right\vert },
\]
which in some sense eliminates enthusiastic outliers in the population.

To move on from posing a meaningful questions to identifying solutions is
actually quite challenging. For example it is not clear, at the level of
generality that we introduce, that the path $t\rightarrow\tilde{\mu}_{t}$ will
have bounded variation or what Banach space to consider it as a path in even
if it does. This raises another issue - in that solving equations such as this
we require the pair $\left(  \mu_{t},W_{t}\left(  \omega\right)  \right)  $ to
be a rough path which normally requires extra data unless one has good control
on $t\rightarrow\mu_{t}$ so we have to make some compromises. No doubt there
is much that can be refined and taken further.

Let $t\rightarrow\mu_{t}$ be a path in the space of probability measures
representing a putative evolution of the population $Y_{t}\left(
\omega\right)  .$ We introduce the "occupation" measure process $\Gamma
_{t}:=\int_{0}^{t}\mu_{t}dt$ and note that it is monotone increasing and
Lipschitz with norm one in the total variation norm on measures. Let
$\theta\left(  y\right)  \mu=\int\phi\left(  y,y^{\prime}\right)  \mu\left(
dy^{\prime}\right)  ,$ then $\theta$ can be viewed as a linear map from our
space of measures to vector fields on the $Y-$space.

We make two significant simplifications to make the problem more tractable:

\begin{enumerate}
\item We only allow so-called weak interactions between the individual and the
population which take place only in the drift component of the equation i.e.
$\xi\left(  Y_{t}\left(  \omega\right)  ,\mathbb{\nu}_{\ast}Y_{t}\right)
$=$\xi\left(  Y_{t}\left(  \omega\right)  \right)  .$

\item The interaction between the individual and the population admits a
superposition principle.
\end{enumerate}

Together these imply that we can write the interaction between $Y,$ its
preferences $W$ and the distribution $\nu$ of the community in the following
form%
\[
dY_{t}\left(  \omega\right)  =\theta\left(  Y_{t}\left(  \omega\right)
\right)  d\Gamma_{t}+\xi\left(  Y_{t}\left(  \omega\right)  \right)
dX_{t}\left(  \omega\right)  .
\]
We then look for fixed points of the map that takes $t\rightarrow\mu_{t}$ to
$t\rightarrow\tilde{\mu}_{t}.$ Since $\Gamma$ has bounded variation this
equation poses fewer technical problems than the general case but still allows
discussion of the Vlasov type problems discussed initially.

The paper is structures as follows. In Section \ref{rp} we spend some time
setting up our notation for the rough path framework; this is the mathematical
technology we use to model the community. Section \ref{finite} then explores
the special case where the law of the preferences is given by a
finitely-supported probability measure on the space of (geometric) rough
paths. Here we prove an existence and uniqueness theorem for the law of the
nonlinear McKean-Vlasov RDE. The methodology here is distinct from that used
later to prove the (more general) result for non-discrete measures. But this
simple case allows us to see very clearly how the weak interaction assumption,
combined with the LV Extension Theorem of \cite{LV} gives rise to the
uniqueness of fixed points. In Section \ref{fps} we proceed with the roadmap
sketched out above. We first present some Gronwall inequalities for rough
differential equations, developing the deterministic estimates from \cite{FV}
and focusing particularly on the conditions needed to ensure the integrability
of the estimates. We make use of the recent paper \cite{CLL} in showing that
these conditions are satisfied for a wide range of preference measures. We
then present conditions that ensure the existence and uniqueness of fixed
points, and discuss the their continuity in the measure on preferences.
Finally, Section \ref{propa} establishes propagation of chaos (\`{a} la
Sznitman \cite{sznitman}) for the convergence of the finite particle system.
We note that this paper has already lead to follow-up work (see, e.g.,
\cite{ball}); we discuss other possible applications of our results.

\section{Preliminaries on rough path theory\label{rp}}

\label{Sect:PRP}There are now a wealth of resources on rough path theory, e.g.
\cite{car}, \cite{FV},\cite{FH},\cite{LQ02} . Rather than give an overview, we
will focus on the notation we need for the current application and direct the
reader to references where appropriate. We first recall the notion of
the\textit{ truncated signature} of a parameterised path in $C^{1-var}\left(
\left[  0,T\right]  ,%
\mathbb{R}
^{d}\right)  $ (the set of continuous paths of bounded variation), this is:
\[
S_{N}\left(  x\right)  _{s,t}:=1+\sum_{k=1}^{N}\int_{s<t_{1}<t_{2}%
<....<t_{k}<t}dx_{t_{1}}\otimes dx_{t_{2}}\otimes...\otimes dx_{t_{k}}\in
T^{N}\left(
\mathbb{R}
^{d}\right)  .
\]
Where $T^{N}\left(
\mathbb{R}
^{d}\right)  =\oplus$ $_{i=0}^{N}\left(
\mathbb{R}
^{d}\right)  ^{\otimes i}$denotes the truncated tensor algebra. We use
$\pi_{n}$ to denote the canonical projection
\[
\pi_{n}:T^{N}\left(
\mathbb{R}
^{d}\right)  \rightarrow\left(
\mathbb{R}
^{d}\right)  ^{\otimes n},\text{ }n=0,1,...,N.
\]
For $\mathbf{x}^{n}$ in $\left(
\mathbb{R}
^{d}\right)  ^{\otimes n}$ we define $\mathbf{x}^{n;\left(  i_{1}%
,...,i_{n}\right)  }$ to be the real number%
\[
\mathbf{x}^{n;\left(  i_{1},...,i_{n}\right)  }=\left(  e_{i_{1}}^{\ast
}\otimes...\otimes e_{i_{n}}^{\ast}\right)  \left(  \mathbf{x}^{n}\right)
=:\left\langle e_{\left(  i_{1},...,i_{n}\right)  }^{\ast},\mathbf{x}%
^{n}\right\rangle ,
\]
where $e_{1}^{\ast},...,e_{d}^{\ast}$ denote the standard dual basis vectors.
We equip each $\left(
\mathbb{R}
^{d}\right)  ^{\otimes n}$ with a compatible tensor norm $\left\vert
\mathbf{\cdot}\right\vert _{\left(
\mathbb{R}
^{d}\right)  ^{\otimes n}}$, and let%
\[
d_{N}\left(  \mathbf{g,h}\right)  :=\max_{i=1,...,N}\left\vert \pi_{i}\left(
\mathbf{g-h}\right)  \right\vert _{\left(
\mathbb{R}
^{d}\right)  ^{\otimes i}}.
\]

It is a well-known that the path $S_{N}\left(  x\right)  $ in fact takes
values in the step-$N$ free nilpotent group with $d$ generators, which we
denote $G^{N}\left(
\mathbb{R}
^{d}\right)  $. Motivated by this, we may consider the set of such
group-valued paths%

\[
\mathbf{x}_{t}=\left(  1,\mathbf{x}_{t}^{1},...,\mathbf{x}_{t}^{\lfloor
p\rfloor}\right)  \in G^{\lfloor p\rfloor}\left(
\mathbb{R}
^{d}\right)  ,
\]
for $p\geq1.$We can then describe the set of "norms" on $G^{\lfloor p\rfloor
}\left(
\mathbb{R}
^{d}\right)  $ which are \textit{homogeneous} with respect to the natural
scaling operation on the tensor algebra (see \cite{FV} for definitions and
details). The subset of these so-called homogeneous norms which are symmetric
and sub-additive (\cite{FV}) give rise to genuine metrics on $G^{\lfloor
p\rfloor}\left(
\mathbb{R}
^{d}\right)  .$ And these metrics\ in turn give rise to the notion of a
homogeneous $p$-variation metric $d_{p\text{-var}}$ on the $G^{\lfloor
p\rfloor}\left(
\mathbb{R}
^{d}\right)  $-valued paths, a typical example being the Carnot-Caratheodory
(CC) metric $d_{CC}$. The group structure provides a natural notion of
increment, namely $\mathbf{x}_{s,t}:=\mathbf{x}_{s}^{-1}\otimes\mathbf{x}_{t}$
and we may then define
\begin{equation}
d_{p-\text{var;}\left[  0,T\right]  }\left(  \mathbf{x,y}\right)  :=\left\vert
\left\vert \mathbf{x-y}\right\vert \right\vert _{p\text{-var;}\left[
0,T\right]  }:=\left(  \sup_{D=\left(  t_{j}\right)  }\sum_{j:t_{j}\in
D}d_{CC}\left(  \mathbf{x}_{t_{j},t_{j+1}},\mathbf{y}_{t_{j},t_{j+1}}\right)
\right)  ^{1/p}. \label{homogeneous norm}%
\end{equation}
If (\ref{homogeneous norm}) is finite then, $\omega_{\mathbf{x}}\left(
s,t\right)  :=\left\vert \left\vert \mathbf{x}\right\vert \right\vert
_{p\text{-var;}\left[  s,t\right]  }^{p}(:=\left\vert \left\vert
\mathbf{x-1}\right\vert \right\vert _{p\text{-var;}\left[  s,t\right]  }^{p}$)
is a control\footnote{i.e. it is a continuous, non-negative, super-additive
function on the simplex $\Delta_{\lbrack0,T]}={(s,t):0=s<=t=T}$ which vanishes
on the diagonal.}. Also of interest will be the \textit{inhomogeneous rough
path metric} defined by%
\[
\rho_{p-var;\left[  0,T\right]  }\left(  \mathbf{x},\mathbf{y}\right)
:=\left\vert \mathbf{x}_{0}-\mathbf{y}_{0}\right\vert _{T^{\lfloor p\rfloor
}\left(
\mathbb{R}
^{d}\right)  }+\max_{i=1,...,\left\lfloor p\right\rfloor }\sup_{D=\left(
t_{j}\right)  }\left(  \sum_{j:t_{j}\in D}\left\vert \pi_{i}\left(
\mathbf{x}_{t_{j},t_{j+1}}-\mathbf{y}_{t_{j},t_{j+1}}\right)  \right\vert
_{\left(
\mathbb{R}
^{d}\right)  ^{\otimes i}}^{p/i}\right)  ^{i/p}.
\]
And the $\omega-$modulus inhomogeneous metric,\ with respect to a fixed
control $\omega,$which is defined by
\[
\rho_{p-\omega;\left[  0,T\right]  }\left(  \mathbf{x},\mathbf{y}\right)
=\left\vert \mathbf{x}_{0}-\mathbf{y}_{0}\right\vert _{T^{\lfloor p\rfloor
}\left(
\mathbb{R}
^{d}\right)  }+\max_{i=1,...,\left\lfloor p\right\rfloor }\frac{\left\vert
\pi_{i}\left(  \mathbf{x}_{s,t}-\mathbf{y}_{s,t}\right)  \right\vert }%
{\omega\left(  s,t\right)  ^{1/p}}.
\]

The space of \textit{weakly geometric} $p-$rough paths will be denoted
$WG\Omega_{p}\left(
\mathbb{R}
^{d}\right)  $. This is the set of continuous paths with values in $G^{\lfloor
p\rfloor}\left(
\mathbb{R}
^{d}\right)  $ (parametrised over some, usually implicit, time interval) such
that (\ref{homogeneous norm}) is finite. A refinement of this notion is the
space of geometric $p-$rough paths, denoted $G\Omega_{p}\left(
\mathbb{R}
^{d}\right)  $, which is the closure of
\[
\left\{  S_{\left\lfloor p\right\rfloor }\left(  x\right)  _{0,\cdot}:x\in
C^{1-var}\left(  \left[  0,T\right]  ,%
\mathbb{R}
^{d}\right)  \right\}
\]
with respect to the rough path metric $d_{p-\text{var}}.$

We will often end up considering an RDE driven by a path $\mathbf{x}$ in
$WG\Omega_{p}\left(
\mathbb{R}
^{d}\right)  $ along a collection of vector fields $V=\left(  V^{1}%
,...,V^{d}\right)  $ on $%
\mathbb{R}
^{e}$. And from the point of view of existence and uniqueness results, the
appropriate way to measure the regularity of the $V_{i}$s results turns out to
be the notion of Lipschitz-$\gamma$ (or, simply, Lip-$\gamma$) in the sense of
Stein. This notion provides a norm on the space of such vector fields, which
we denote $\left\vert \cdot\right\vert _{Lip-\gamma}.$ We will often make use
of the shorthand
\[
\left\vert V\right\vert _{Lip-\gamma}=\max_{i=1,...,d}\left\vert
V_{i}\right\vert _{Lip-\gamma}.
\]

Finally, throughout the article we will consider spaces of probabilities
measure on metric spaces $\left(  S,d\right)  .$

\begin{notation}
We will use $\mathcal{M}\left(  S\right)  $ to denote the space of probability
measures on $\left(  S,\mathcal{B}\left(  S\right)  \right)  .$ For $p>0,$
$\mathcal{M}_{p}\left(  S\right)  $ will represent the subset of
$\mathcal{M}\left(  S\right)  $ which have finite $p^{th}$-moment in the sense
that
\[
\int_{S}d\left(  s_{0},s\right)  ^{p}\mu\left(  ds\right)  <\infty,
\]
for some (and hence every) $s_{0}\in S$
\end{notation}

It will be convenient to have a shorthand notation for some of these spaces.

\begin{notation}
\label{not}We will write
\[
\left(  S_{N,e},\sigma_{N}\right)  \text{ for }\left(  G^{N}\left(
\mathbb{R}
^{e}\right)  ,d_{N}\right)  ,\text{ and }\left(  P_{p,e},\rho_{p}\right)
\text{ for }\left(  G\Omega_{p}\left(
\mathbb{R}
^{d}\right)  ,\rho_{p-var;\left[  0,T\right]  }\right)  .
\]
Furthermore $\left(  S_{t,e},\sigma_{t}\right)  $ will mean $\left(
S_{\lfloor t\rfloor,e},\sigma_{\lfloor t\rfloor}\right)  $ whenever $t$ is not
an integer.
\end{notation}

\section{\bigskip Weakly interacting communties\label{finite}}

Let $\left(  \mu_{t}\right)  _{t\in\left[  0,T\right]  }$ be a family of
probability measures on $G^{\lfloor p\rfloor}\left(
\mathbb{R}
^{d}\right)  $ parameterised by time. The main object of study in this paper
will be solutions to rough differential equations which incorporate
\textit{weak mean-field interactions }with $\left(  \mu_{t}\right)
_{t\in\left[  0,T\right]  }.$ By this we mean equations of the following type
\begin{equation}
d\mathbf{y}_{t}=\int_{G^{\lfloor p\rfloor}\left(
\mathbb{R}
^{d}\right)  }\sigma\left(  y_{t},\pi_{1}\mathbf{y}\right)  \mu_{t}\left(
d\mathbf{y}\right)  dt+V\left(  y_{t}\right)  d\mathbf{x,}\text{ }y\left(
0\right)  =y_{0}. \label{mf}%
\end{equation}
The rough path $\mathbf{x}$ flows along the vector fields $V=\left(
V^{1},...,V^{d}\right)  ,$ but the resulting trajectory is also influenced
by\ $\left(  \mu_{t}\right)  _{t\in\left[  0,T\right]  }$ through the
\textit{interaction kernel} $\sigma.$ Assuming enough regularity on the path
$t\mapsto\mu$ we may define the integral
\[
\gamma_{t}^{\mu}=\int_{0}^{t}\mu_{s}ds,
\]
a continuous bounded variation path in an appropriately chosen ambient Banach
space. It is convenient to rewrite the main equation (\ref{mf}) as%
\begin{equation}
d\mathbf{y}_{t}=V^{0}\left(  y_{t}\right)  d\gamma_{t}^{\mu}+V\left(
y_{t}\right)  d\mathbf{x}_{t}+V\left(  y_{t}\right)  d\mathbf{x}_{t},\text{
}y\left(  0\right)  =y_{0}, \label{IRDE}%
\end{equation}
where $V^{0}$ and $\sigma$ are related by%
\[
V^{0}\left(  y\right)  \left(  \mu\right)  =\int_{G^{\lfloor p\rfloor}\left(
\mathbb{R}
^{d}\right)  }\sigma\left(  y,\pi_{1}\mathbf{y}\right)  \mu\left(
d\mathbf{y}\right)  .
\]
We will discuss the detail of this construction in Section \ref{fps}. In the
cases we consider, $\left(  \mu_{t}\right)  _{t\in\left[  0,T\right]  }$ will
be derived from the marginal distributions of a probability measure in
$\mathcal{M}\left(  P_{p,e}\right)  ;$ i.e. those derived from pushing-forward
under the evaluation maps $\psi_{t}\left(  \mathbf{x}\right)  =\mathbf{x}_{t}%
$, $t\in\left[  0,T\right]  .$ We denote a solution to (\ref{IRDE}) by
$\Theta_{V^{0},V}\left(  \mu,y_{0},\mathbf{x}\right)  ,$ and fix a probability
measure $u_{0}\times\nu$ on $%
\mathbb{R}
^{e}\times G\Omega_{p}\left(
\mathbb{R}
^{d}\right)  .$ $u_{0}$ describes the initial configuration of the particles
and $\nu$ is the law of the preferences or, more conveniently, the
\textit{preference measure.} By taking a realisation $\left(  Y_{0}%
,\mathbf{X}\right)  $ of $u_{0}\times\nu$ on some probability space $\left(
\Omega,\mathcal{F},P\right)  ,$ and then using $\mathbf{X}$ to solve
(\ref{IRDE}) we will have constructed a well-defined map $\Psi_{\nu}$ from the
space $\mathcal{M}\left(  G\Omega_{p}\left(
\mathbb{R}
^{e}\right)  \right)  $ to itself given by the push-forward$:$
\[
\Psi_{\nu}:\mu\mapsto\left[  \Theta_{V^{0},V}\left(  \mu,\cdot,\cdot\right)
\right]  _{\ast}\left(  u_{0}\times\nu\right)  .
\]
$\mu$ will then be fixed point of this map if and only if $\Theta_{V^{0}%
,V}\left(  \mu,Y_{0},\mathbf{X}\right)  $ is a solution the (nonlinear)
McKean-Vlasov-type RDE%
\begin{equation}
\left\{
\begin{array}
[c]{c}%
d\mathbf{Y}_{t}=V\left(  Y_{t}^{\mu}\right)  d\mathbf{X}_{t}+V^{0}\left(
Y_{t}^{\mu}\right)  d\gamma_{t}^{\mu}\\
\text{Law}\left(  \mathbf{Y}\right)  =\mu,\text{ Law}\left(  Y_{0}\right)
=u_{0}%
\end{array}
\right.  . \label{mckv}%
\end{equation}
A key objective of this paper is to demonstrate that there exist unique fixed
points to (\ref{mckv}) for a class of preference measures which extend far
beyond the usual semimartingale setting.

We first spend time developing an important special case, namely when $\nu$ is
a finitely-supported discrete measure of the form%
\[
\nu=\sum_{i=1}^{N}\lambda_{i}\delta_{\mathbf{x}_{i}}\in\mathcal{M}\left(
P_{p,d}\right)  .
\]
In this setting, we can attempt to resolve the fixed-point-problem
(\ref{mckv}) by solving the system of RDEs%
\begin{equation}
d\mathbf{y}_{i}\left(  t\right)  =V\left(  y_{i}\left(  t\right)  \right)
d\mathbf{x}_{i}\left(  t\right)  +\sum_{j=1}^{N}\lambda_{j}\sigma\left(
y_{i}\left(  t\right)  ,y_{j}\left(  t\right)  \right)  d\gamma_{t}^{\mu
},\text{ \ }y_{i}\left(  0\right)  =y_{i}\label{system}%
\end{equation}
for $i=1,....,N.$ And then defining the measure to be the convolution
\[
\mu=u_{0}^{\otimes N}\ast\left(  \sum_{i=1}^{N}\lambda_{i}\delta
_{\mathbf{y}_{i}}\right)  ,
\]
where $u_{0}^{\otimes n}$ is the $n$-fold product measure of $u_{0}$. More
precisely this means that
\[
\mu\left(  A\right)  =\sum_{i=1}^{N}\lambda_{i}\int_{%
\mathbb{R}
^{e}\times...\times%
\mathbb{R}
^{e}}\delta_{\mathbf{y}_{i}^{y_{i}}}\left(  A\right)  u_{0}\left(
dy_{1}\right)  ...u_{0}\left(  dy_{N}\right)  ,\text{ }\forall A\in
\mathcal{B}\left(  P_{p,e}\right)
\]
where we have written $\mathbf{y}_{i}^{y_{i}}$ to emphasise the dependence of
$\mathbf{y}_{i}$ on its starting point $y_{i}$. With $\mu$ defined in this way
we would expect that $\Psi_{\nu}\left(  \mu\right)  =\mu,$ and indeed this
approach will work for smooth preferences. In the rough case $\left(
p\geq2\right)  $ however things are more complex. Here in order to solve
(\ref{system}) we need to define \textit{a priori} the cross-iterated
integrals between the (components of) the preferences $\mathbf{x}^{i}$ and
$\mathbf{x}^{j}.$ The LV Extension Theorem (\cite{LV}) guarantees that this
can always be done, but in general there are many choices for the extension.
To ensure uniqueness of the fixed point, we need to check that the resulting
solution is not sensitive to this choice; the remainder of this section will
present conditions which will guarantee this.

The results of this section will later be subsumed by the general fixed point
theorem of Section \ref{fps}. Nonetheless they are important for three
reasons. Firstly they expose, in an original and lucid way, the importance of
the weakly interacting structure; secondly, they highlight the main obstacle
in extending the analysis to general interactions, in a way that cannot be
easily discerned from the general fixed point result; thirdly, they crucially
underlie our later treatment of the convergent behaviour of the finite
particle system.

\subsection{A two-particle system}

To make clear the structure of the argument, we first deal with the case where
$N=2$ and $p\in\left(  2,3\right)  ;$ i.e. the preference measure is supported
on only two geometric rough paths in $G\Omega_{p}\left(
\mathbb{R}
^{d}\right)  .$ We write $\nu=\lambda\delta_{\mathbf{x}_{1}}+\left(
1-\lambda\right)  \delta_{\mathbf{x}_{2}}.$ By the LV Extension theorem there
exists an element $\mathbf{x}$ in $WG\Omega_{p}\left(
\mathbb{R}
^{2d}\right)  $ which lifts $\left(  x_{1},x_{2}\right)  $ consistently with
$\mathbf{x}_{1}$ and $\mathbf{x}_{2}$ in the sense that,%
\[
P_{j}\mathbf{x}^{1}=\mathbf{x}_{j}^{1}\text{ and }\left(  P_{j}\otimes
P_{j}\right)  \mathbf{x}^{2}=\mathbf{x}_{j}^{2}\text{ for }j=1,2
\]
where $P_{1},P_{2}:%
\mathbb{R}
^{2e}\cong%
\mathbb{R}
^{e}\times%
\mathbb{R}
^{e}\rightarrow%
\mathbb{R}
^{e}$ are defined by $P_{1}z=x$ and $P_{2}z=y$ when $z=\left(  x,y\right)  ,$
and where $\left(  P_{j}\otimes P_{j}\right)  \left(  z_{1}\otimes
z_{2}\right)  =P_{j}z_{1}\otimes P_{j}z_{2}\in\left(
\mathbb{R}
^{2e}\right)  ^{\otimes2}.$ We can simplify this by writing
\begin{equation}
\mathbf{x}^{1}=\left(  \mathbf{x}_{1}^{1},\mathbf{x}_{2}^{1}\right)  \in%
\mathbb{R}
^{2d},\text{ \ }\mathbf{x}^{2}=\left(
\begin{array}
[c]{cc}%
\mathbf{x}_{1}^{2} & \ast\\
\ast & \mathbf{x}_{2}^{2}%
\end{array}
\right)  \in\left(
\mathbb{R}
^{2d}\right)  ^{\otimes2},\label{lift}%
\end{equation}
under the obvious identifications. The only constraint on the terms $(\ast)$
arises from the need to make $\mathbf{x}$ weakly geometric. Given such an
extension, we can solve the following RDE uniquely%
\begin{equation}
d\mathbf{y}_{t}=W^{0}\left(  y_{t}\right)  dt+W\left(  y_{t}\right)
d\mathbf{x}_{t},\text{ }y\left(  0\right)  =\left(  y_{1}\left(  0\right)
,y_{1}\left(  0\right)  \right)  \in%
\mathbb{R}
^{2e}.\label{sys}%
\end{equation}
Wherein $W=\left(  W^{1},...,W^{2d}\right)  $ is the collection of vector
fields on $%
\mathbb{R}
^{2e}\cong%
\mathbb{R}
^{e}\times%
\mathbb{R}
^{e}$ defined by%
\begin{align}
W^{i}\left(  y_{1},y_{2}\right)   &  =\left(  V^{i}\left(  y_{1}\right)  ,0_{%
\mathbb{R}
^{e}}\right)  ^{t},\text{ }i=1,....,d\label{vfs}\\
W^{i}\left(  y_{1},y_{2}\right)   &  =\left(  0_{%
\mathbb{R}
^{e}},V^{i}\left(  y_{2}\right)  \right)  ^{t},\text{ }i=d+1,....,2d,\nonumber
\end{align}
and the interaction is transmitted through
\[
W^{0}\left(  y_{1},y_{2}\right)  =\lambda\left(  \sigma\left(  y_{1}%
,y_{1}\right)  ,\sigma\left(  y_{2},y_{1}\right)  \right)  ^{t}+\left(
1-\lambda\right)  \left(  \sigma\left(  y_{1},y_{2}\right)  ,\sigma\left(
y_{2},y_{2}\right)  \right)  ^{t}.
\]
By writing the solution $\mathbf{y}$ in terms of its projections
\begin{equation}
\mathbf{y}^{1}=\left(  \mathbf{y}_{1}^{1},\mathbf{y}_{2}^{1}\right)  \in%
\mathbb{R}
^{2e},\text{ }\mathbf{y}^{2}=\left(
\begin{array}
[c]{cc}%
\mathbf{y}_{1}^{2} & \ast\\
\ast & \mathbf{y}_{2}^{2}%
\end{array}
\right)  \in\left(
\mathbb{R}
^{2e}\right)  ^{\otimes2},\label{proj}%
\end{equation}
we can obtain $\mathbf{y}_{i}=\left(  1,\mathbf{y}_{i}^{1},\mathbf{y}_{i}%
^{2}\right)  \in WG\Omega_{p}\left(
\mathbb{R}
^{e}\right)  .$ We will prove that the probability measure
\begin{equation}
\mu=u_{0}^{\otimes2}\ast\left[  \lambda\delta_{\mathbf{y}_{1}}+\left(
1-\lambda\right)  \delta_{\mathbf{y}_{2}}\right]  \in\mathcal{M}\left(
P_{p,e}\right)  \label{fp}%
\end{equation}
is a fixed point of the map $\Psi_{\nu}.$ We will then show that every fixed
point has the form (\ref{fp}); i.e. its suppport is $\left\{  \mathbf{y}%
_{1},\mathbf{y}_{2}\right\}  ,$ where $\mathbf{y}_{1},\mathbf{y}_{2}$ are
projections of \textit{the} solution to (\ref{sys}) driven by \textit{any}
extension $\mathbf{x.}$ The uniqueness of the fixed point will follow by
proving that the projections $\mathbf{y}_{1}$ and $\mathbf{y}_{2}$ do not
depend on the extension (and hence neither does the measure (\ref{fp})). This
is the essential content of the following proposition.

\begin{proposition}
\label{2part}Let $2<p<3$ and $y_{1}\left(  0\right)  ,y_{2}\left(  0\right)
\in%
\mathbb{R}
^{e}.$ Suppose that $\mathbf{x}_{1}$ and $\mathbf{x}_{2}$ are two elements of
$G\Omega_{p}\left(
\mathbb{R}
^{d}\right)  .$ Assume that $W^{0}$ $\ $and $W=\left(  W^{1},...,W^{2d}%
\right)  $ are, respectively, vector fields in $Lip^{\beta}\left(
\mathbb{R}
^{2e}\right)  $ and $Lip^{\gamma}\left(
\mathbb{R}
^{2e}\right)  $ for some $\beta>1$ and $\gamma>p$ $.$ Let $\mathbf{x}$ be any
element of $WG\Omega_{p}\left(
\mathbb{R}
^{2d}\right)  $ which extends $\mathbf{x}_{1}$ and $\mathbf{x}_{2}$ in the
sense of (\ref{lift}), and let $\mathbf{y}$ be the unique solution in
$WG\Omega_{p}\left(
\mathbb{R}
^{2e}\right)  $ in to the RDE (\ref{sys}) driven by $\mathbf{x.}$ Then
$\mathbf{y}$ has the property that its projections $\mathbf{y}_{1}%
,\mathbf{y}_{2}$ (as given in (\ref{proj})) are elements of $WG\Omega
_{q}\left(
\mathbb{R}
^{e}\right)  $ which depend on $\mathbf{x}_{1}$ and $\mathbf{x}_{2},$ but not
on the extension $\mathbf{x.}$
\end{proposition}

\begin{proof}
We prove that $\mathbf{y}_{1}$ and $\mathbf{y}_{2},$ the projections of the
solution to (\ref{sys}), depend only on $\mathbf{x}_{1}$ and $\mathbf{x}_{2}$
and not the iterated integral between them. In other words, that
$\mathbf{y}_{1}$ and $\mathbf{y}_{2}$ have meaning independently of the terms
$\ast$ needed to specify the joint lift in (\ref{lift}). To see this we recall
(\cite{FH}) that $y_{s,t}$, the increment of the path level solution over
$\left[  s,t\right]  $, is equal to
\begin{equation}
\lim_{\left\vert D\left[  s,t\right]  \right\vert \rightarrow0}\sum
_{i:t_{i}\in D\left[  s,t\right]  }\left[  W^{0}\left(  y_{t_{i}}\right)
\left(  t_{i+1}-t_{i}\right)  +W\left(  y_{t_{i}}\right)  \mathbf{x}%
_{t_{i},t_{i+1}}^{1}+DW\left(  y_{t_{i}}\right)  W\left(  y_{t_{i}}\right)
\mathbf{x}_{t_{i},t_{i+1}}^{2}\right]  , \label{approx}%
\end{equation}
where $DW\left(  y_{t_{i}}\right)  W\left(  y_{t_{i}}\right)  \left[  x\otimes
y\right]  =DW\left(  y_{t_{i}}\right)  \left[  W\left(  y_{t_{i}}\right)
x\right]  \left[  y\right]  ,$ and $D\left[  s,t\right]  $ denotes a partition
of $\left[  s,t\right]  .$ The last term in the summands in (\ref{approx})
equals%
\[
\frac{1}{2}DW\left(  y_{t_{i}}\right)  W\left(  y_{t_{i}}\right)  \left[
\mathbf{x}_{t_{i},t_{i+1}}^{1}\otimes\mathbf{x}_{t_{i},t_{i+1}}^{1}\right]
+DW\left(  y_{t_{i}}\right)  W\left(  y_{t_{i}}\right)  \mathbf{x}%
_{t_{i},t_{i+1}}^{2;a},
\]
where $\mathbf{x}_{t_{i},t_{i+1}}^{2;a}$ is the anti-symmetric part of the
$2$-tensor $\mathbf{x}_{t_{i},t_{i+1}}^{2}.$ The first term only depends on
$\mathbf{x}_{t_{i},t_{i+1}}^{1}$, and the second term can be simplified to
\[
\sum_{p,q=1}^{2d}\left[  W^{p},W^{q}\right]  \left(  y_{t_{i}}\right)
\mathbf{x}_{t_{i},t_{i+1}}^{2;\left(  p,q\right)  }.
\]
From the definition of $\left(  W^{i}\right)  $ it is easy to see that the Lie
bracket
\[
\left[  W^{p},W^{q}\right]  \equiv0\text{ }\forall p\in\left\{
1,...,d\right\}  \text{ and }\forall q\in\left\{  d+1,...,2d\right\}
\]
(and, therefore, it also vanishes\ for every $p\in\left\{  d+1,...,2d\right\}
$ and $q\in\left\{  1,...,d\right\}  $ by antisymmetry). Each summand in
(\ref{approx}) thus only depends on $\mathbf{x}^{1}$, $\mathbf{x}_{1}^{2}$ and
$\mathbf{x}_{2}^{2},$ but not on the terms of $\mathbf{x}^{2}$ corresponding
to integrals between $\mathbf{x}_{1}^{1}$ and $\mathbf{x}_{2}^{1}$; the same
is hence true of the limit, $y_{s,t}.$

We recall that $\mathbf{y}_{s,t}^{2}$ is the limit as $\left\vert D\left[
s,t\right]  \right\vert \rightarrow0$ of
\begin{equation}
\sum_{i:t_{i}\in D\left[  s,t\right]  }\left[  y_{s,t_{i+1}}\otimes
y_{t_{i},t_{i+1}}+\left[  W\left(  y_{t_{i}}\right)  \otimes W\left(
y_{t_{i}}\right)  \right]  \mathbf{x}_{t_{i},t_{i+1}}^{2}\right]  ,
\label{sec}%
\end{equation}
so that in general $\mathbf{y}^{2}$ \textit{does} depend on the extension.
However, by taking projections the dependence disappears. To see this just let
$P_{1}:%
\mathbb{R}
^{e}\times%
\mathbb{R}
^{e}\rightarrow%
\mathbb{R}
^{e}$ $\ $denote the projection $P_{1}\left(  y_{1},y_{2}\right)  =y_{1}$, so
that $\mathbf{y}_{1}^{2}=\left(  P_{1}\otimes P_{1}\right)  \left(
\mathbf{y}^{2}\right)  .$\ We then observe that
\begin{equation}
\left(  P_{1}\otimes P_{1}\right)  \left(  W\left(  y_{t}\right)  \otimes
W\left(  y_{t}\right)  \right)  \mathbf{=}\left(  V\left(  y_{t}\right)
,0\right)  \otimes\left(  V\left(  y_{t}\right)  ,0\right)  , \label{ind}%
\end{equation}
and also the corresponding relation for $\mathbf{y}_{2}^{2}.$ The claim then
follows at once from (\ref{sec}) and (\ref{ind}).
\end{proof}

\begin{corollary}
\label{wd}Let $\nu=\lambda\delta_{\mathbf{x}_{1}}+\left(  1-\lambda\right)
\delta_{\mathbf{x}_{2}}\in\mathcal{M}\left(  P_{p,d}\right)  $. There exists a
unique fixed point $\mu\in\mathcal{M}_{1}\left(  P_{p,e}\right)  $ of the map
$\Psi_{\nu}$ which is given explicitly by
\begin{equation}
\mu=u_{0}^{\otimes2}\ast\left[  \lambda\delta_{\mathbf{y}_{1}}+\left(
1-\lambda\right)  \delta_{\mathbf{y}_{2}}\right]  , \label{fp1}%
\end{equation}
where $\mathbf{y}_{i}\in WG\Omega_{p}\left(
\mathbb{R}
^{2d}\right)  ,$ $i=1,2$ are the projections of the solution to (\ref{sys})
driven by any extension $\mathbf{x}$ of $\mathbf{x}_{1}$ and $\mathbf{x}_{2}.$
\end{corollary}

\begin{proof}
The previous proposition ensures (\ref{fp1}) is well-defined. In other words,
for every fixed realisation $\left(  y_{1}\left(  0\right)  ,y_{2}\left(
0\right)  \right)  $ of $u_{0}^{\otimes2}$ the rough paths $\mathbf{y}_{1}$
and $\mathbf{y}_{2}$ will not depend on the choice of extension. We muct check
that this is the only fixed point. To do so, first note that the assumption on
$\nu$ implies that any fixed point must have the form
\[
\mu=u_{0}^{\otimes2}\ast\left[  \lambda\delta_{\mathbf{z}_{1}}+\left(
1-\lambda\right)  \delta_{\mathbf{z}_{2}}\right]  ,
\]
where $\mathbf{z}_{i}$, $i=1,2$ are elements of $G\Omega_{p}\left(
\mathbb{R}
^{e}\right)  .$ Then, by the definition of the map $\Psi_{\nu},$
$\mathbf{z}_{i}$, $i=1,2$ must solve the RDEs
\[
d\mathbf{z}_{i}\left(  t\right)  =V^{0}\left(  y_{i}\left(  t\right)  \right)
d\gamma_{t}^{\mu}+V\left(  y_{i}\left(  t\right)  \right)  d\mathbf{x}%
_{i}\left(  t\right)  \text{ }.
\]
Let $\mathbf{x}$ be any path in $WG\Omega_{p}\left(
\mathbb{R}
^{2d}\right)  $ whose projections are consistent with $\mathbf{x}_{1}$ and
$\mathbf{x}_{2}$. Then, since $\mathbf{z}_{1}$ and $\mathbf{z}_{2}$ may both
be written as solutions to RDEs driven by $\mathbf{x,}$ we may define in a
canonical way (see \cite{FH}) a path $\mathbf{z}$ in $WG\Omega_{p}\left(
\mathbb{R}
^{2d}\right)  ,$ which has $\mathbf{z}_{1}$ and $\mathbf{z}_{2}$ as its
projections$.$ $\mathbf{z}$ is then the solution of the RDE (\ref{sys}) driven
along $\mathbf{x}$ $.$
\end{proof}

\subsection{N-particle systems}

We will later want to consider the propagation of choas phenomenon for rough
differential equations, and this requires us to present the treatment of the
previous subsection for a population of particles of arbitrary finite size
$N.$ We therefore suppose that the preference measure is now given by
\[
\nu=\sum_{i=1}^{N}\lambda_{i}\delta_{\mathbf{x}_{i}}.
\]
Analogously to the two-particle case (recall (\ref{vfs})) we define vector
fields $W^{0}$ and $W=\left(  W^{1},...,W^{Nd}\right)  $ on $%
\mathbb{R}
^{Ne}\cong%
\mathbb{R}
^{e}\times....\times%
\mathbb{R}
^{e}$ (this time as differential operators for notational ease) by writing
$y=\left(  y^{1},....,y^{N}\right)  \in%
\mathbb{R}
^{Ne}\cong%
\mathbb{R}
^{e}\times....\times%
\mathbb{R}
^{e}$ and setting
\begin{equation}
W^{0}\left(  y\right)  =\sum_{m=1}^{N}\sum_{k=1}^{e}\sum_{i=1}^{N}\lambda
_{i}\sigma_{k}\left(  y^{m},y^{i}\right)  \frac{\partial}{\partial y_{k}^{m}}
\label{vf2}%
\end{equation}
and, using the convention $kd=d\left(  \operatorname{mod}d\right)  $ for
$k\in\mathbb{Z}$,%
\begin{equation}
W^{j}\left(  y\right)  =\sum_{k=1}^{e}V_{k}^{j\left(  \operatorname{mod}%
d\right)  }\left(  y^{1+\lfloor\left(  j-1\right)  /d\rfloor}\right)
\frac{\partial}{\partial y_{k}^{1+\lfloor\left(  j-1\right)  /d\rfloor}%
},\text{ }j=1,...,Nd. \label{vf1}%
\end{equation}
As before, we will be interested in rough paths in $WG\Omega_{p}\left(
\mathbb{R}
^{Nd}\right)  $ whose projections contain each of the rough paths
$\mathbf{x}_{1},\mathbf{x}_{2},...,\mathbf{x}_{N},$ which together form the
support of $\nu.$ The following notation indexes the components of the
extension in terms of the components of $\mathbf{x}_{1},\mathbf{x}%
_{2},...,\mathbf{x}_{N}$.

\begin{notation}
For each $k$ in $\mathbb{N}\ $and $m=0,....,N-1$ define $I_{k,m}$, a subset of
$\left\{  1,....,Nd\right\}  ^{k},$ by
\[
\left(  i_{1},...,i_{k}\right)  \in I_{k,m}\text{ iff }\left\{  i_{1}%
,...,i_{k}\right\}  \subseteq\left\{  md+1,...,\left(  m+1\right)  d\right\}
.
\]
We will let $I_{k}$ denote the subset $\cup_{m=0}^{N-1}I_{k,m}.$
\end{notation}

We now formalise the precise sense in which $\left\{  \mathbf{x}%
_{1},\mathbf{x}_{2},...,\mathbf{x}_{N}\right\}  $ is related to any extension.

\begin{definition}
If $\left\{  \mathbf{x}_{1},\mathbf{x}_{2},...,\mathbf{x}_{N}\right\}  $ is a
collection of rough paths in $G\Omega_{p}\left(
\mathbb{R}
^{d}\right)  ,$ then we say that $\mathbf{x}$ in $WG\Omega_{p}\left(
\mathbb{R}
^{Nd}\right)  $ is a lift which is consistent with $\mathbf{x}_{1}%
,\mathbf{x}_{2},...,\mathbf{x}_{N},$ if for every $k=1,...,\lfloor p\rfloor$
its projections satisfy
\[
\pi_{k}^{\left(  i_{1},...,i_{k}\right)  }\left(  \mathbf{x}\right)
=\mathbf{x}_{m}^{k;\left(  i_{1}\left(  \operatorname{mod}d\right)
,...,i_{k}\left(  \operatorname{mod}d\right)  \right)  },\text{ }%
\forall\left(  i_{1},...,i_{k}\right)  \in I_{k,m},\text{ }\forall m=1,...,N.
\]

\end{definition}

We now chose any lift $\mathbf{x}$ which is consistent with $\mathbf{x}%
_{1},\mathbf{x}_{2},...,\mathbf{x}_{N}$ . We want to show that if we solve the
RDE%
\[
d\mathbf{y}_{t}=W^{0}\left(  y_{t}\right)  dt+W\left(  y_{t}\right)
d\mathbf{x}_{t},\text{ }y\left(  0\right)  =\left(  y_{1}\left(  0\right)
,...,y_{N}\left(  0\right)  \right)  \in%
\mathbb{R}
^{Ne},
\]
along $\mathbf{x,}$ then the output $\mathbf{y}$ will have the same
projections irrespective of the initial choice of lift. To do so, we have to
identify normal subgroup $K$ of $G^{n}\left(
\mathbb{R}
^{Nd}\right)  $ so that $\left\{  \mathbf{x}_{1},\mathbf{x}_{2},...,\mathbf{x}%
_{N}\right\}  $ can be identified with a path in the quotient group
$G^{n}\left(
\mathbb{R}
^{Nd}\right)  /K.$

\begin{lemma}
For $n\in%
\mathbb{N}
$ let $\mathfrak{g}^{n}\left(
\mathbb{R}
^{Nd}\right)  $ denote the Lie algebra of $G^{n}\left(
\mathbb{R}
^{Nd}\right)  .$ Suppose $\mathfrak{k}^{n}\left(
\mathbb{R}
^{Nd}\right)  $ is the subset of $\mathfrak{g}^{n}\left(
\mathbb{R}
^{Nd}\right)  $ defined by
\[
\mathfrak{k}^{n}\left(
\mathbb{R}
^{Nd}\right)  =\left\{  a\in\mathfrak{g}^{n}\left(
\mathbb{R}
^{Nd}\right)  :\left\langle e_{I}^{\ast},a^{\left\vert I\right\vert
}\right\rangle =0,\forall I\in\cup_{k=1}^{n}I_{k}\right\}  ,
\]
where, if $I=\left(  i_{1},...,i_{k}\right)  ,$ we write $\left\vert
I\right\vert =k$ and $e_{I}^{\ast}:=e_{i_{1}}^{\ast}\otimes....\otimes
e_{i_{k}}^{\ast}$. Let%
\[
K^{n}\left(
\mathbb{R}
^{Nd}\right)  :=\exp\left(  \mathfrak{k}^{n}\left(
\mathbb{R}
^{Nd}\right)  \right)  .
\]
Then $K^{n}\left(
\mathbb{R}
^{Nd}\right)  $ is a connected Lie subgroup of $G^{n}\left(
\mathbb{R}
^{Nd}\right)  $, $\mathfrak{k}^{n}\left(
\mathbb{R}
^{Nd}\right)  $ is an ideal in $\mathfrak{g}^{n}\left(
\mathbb{R}
^{Nd}\right)  $ and hence $K^{n}\left(
\mathbb{R}
^{Nd}\right)  $ is a normal subgroup of $G^{n}\left(
\mathbb{R}
^{Nd}\right)  .$
\end{lemma}

\begin{proof}
It is immediate that $K^{n}\left(
\mathbb{R}
^{Nd}\right)  $ is a connected Lie subgroup. To prove that $\mathfrak{k}%
^{n}\left(
\mathbb{R}
^{Nd}\right)  $ is an ideal $\mathfrak{g}^{n}\left(
\mathbb{R}
^{Nd}\right)  $ we need to show that for any $a$ in $\mathfrak{k}^{n}\left(
\mathbb{R}
^{Nd}\right)  $ and $b$ in $\mathfrak{g}^{n}\left(
\mathbb{R}
^{Nd}\right)  $ we have
\[
\left\langle e_{I}^{\ast},\left[  a,b\right]  ^{\left\vert I\right\vert
}\right\rangle =0\text{ for all }I\in I_{k},k=1,...,n.
\]
But this follows by noticing that
\[
\left\langle e_{I}^{\ast},\left(  a\otimes b\right)  ^{k}\right\rangle
=\sum_{l=1}^{k-1}\left\langle e_{I}^{\ast},a^{l}\otimes b^{k-l}\right\rangle
=\sum_{l=1}^{k-1}\left\langle e_{I\left(  l\right)  }^{\ast},a^{l}%
\right\rangle \left\langle e_{I\left(  k-l\right)  }^{\ast},b^{k-l}%
\right\rangle =0.
\]
Where, for every $l=1,....,k-l,$ we have written $I=:\left(  I\left(
l\right)  ,I\left(  k-l\right)  \right)  $ and used the fact that $I\in I_{k}$
to deduce $I\left(  l\right)  \in I_{l}$ and $I\left(  k-l\right)  \in
I_{k-l}.$ It is easily seen from this that $\left\langle e_{I}^{\ast},\left[
a,b\right]  ^{\left\vert I\right\vert }\right\rangle =0.$ The assertion that
$K^{n}\left(
\mathbb{R}
^{Nd}\right)  $ is normal then follows from the well-known correspondence
between ideals of Lie algebras and normal subgroups of the Lie group (see,
e.g., \cite{lee}).
\end{proof}

\begin{remark}
\label{lift1}In a straight forward way we may\ uniquely identify any given
collection of rough paths $\left\{  \mathbf{x}_{1},\mathbf{x}_{2}%
,...,\mathbf{x}_{N}\right\}  $ in $G\Omega_{p}\left(
\mathbb{R}
^{d}\right)  $ with a path, which we denote by $\left(  \mathbf{x}%
_{1},\mathbf{x}_{2},...,\mathbf{x}_{N}\right)  $, in the quotient group:
\[
G^{\lfloor p\rfloor}\left(
\mathbb{R}
^{Nd}\right)  /K^{\lfloor p\rfloor}\left(
\mathbb{R}
^{Nd}\right)  .
\]
$\left(  \mathbf{x}_{1},\mathbf{x}_{2},...,\mathbf{x}_{N}\right)  $ will then
have finite $p-$variation with respect to the homogenous quotient norm (see
\cite{LV}). Any extension $\mathbf{x\in}WG\Omega_{p}\left(
\mathbb{R}
^{Nd}\right)  $ which is consistent with $\mathbf{x}_{1},\mathbf{x}%
_{2},...,\mathbf{x}_{N}$ as described above, will then extend $\left(
\mathbf{x}_{1},\mathbf{x}_{2},...,\mathbf{x}_{N}\right)  $ in the obvious
sense that
\[
\pi_{G^{\lfloor p\rfloor}\left(
\mathbb{R}
^{Nd}\right)  ,G^{\lfloor p\rfloor}\left(
\mathbb{R}
^{Nd}\right)  /K^{\lfloor p\rfloor}\left(
\mathbb{R}
^{Nd}\right)  }\left(  \mathbf{x}\right)  =\left(  \mathbf{x}_{1}%
,\mathbf{x}_{2},...,\mathbf{x}_{N}\right)  .
\]

\end{remark}

We now prove the generalisation of Proposition \ref{2part} to the $N$-particle system.

\begin{theorem}
\label{Npart}Let $p\geq1$, $y_{1}\left(  0\right)  ,...,y_{N}\left(  0\right)
\in%
\mathbb{R}
^{e},$ and suppose that $\left\{  \mathbf{x}_{1},\mathbf{x}_{2},...,\mathbf{x}%
_{N}\right\}  $ is a collection of rough paths in $G\Omega_{p}\left(
\mathbb{R}
^{d}\right)  .$ Assume that $W^{0}$, defined by (\ref{vf2}), $\ $and
$W=\left(  W^{1},...,W^{Nd}\right)  $, defined by (\ref{vf1}) are,
respectively, vector fields in $Lip^{\beta}\left(
\mathbb{R}
^{Ne}\right)  $ and $Lip^{\gamma}\left(
\mathbb{R}
^{Ne}\right)  $ for some $\beta>1$ and $\gamma>p$ $.$ For any $q$ in
$[p,\gamma)$ let $\mathbf{x}$ be an element of $WG\Omega_{q}\left(
\mathbb{R}
^{2d}\right)  $ which extends $\mathbf{x}_{1},\mathbf{x}_{2},...,\mathbf{x}%
_{N}$ in the sense of (\ref{lift1}), and suppose $\mathbf{y}$ be the unique
solution in $WG\Omega_{q}\left(
\mathbb{R}
^{Ne}\right)  $ to the RDE (\ref{sys}) driven by $\mathbf{x.}$ Then
$\mathbf{y}$ has the property that its projections $\mathbf{y}_{1}%
,...,\mathbf{y}_{N}$ to elements of $WG\Omega_{q}\left(
\mathbb{R}
^{e}\right)  $ depend on $\mathbf{x}_{1},\mathbf{x}_{2},...,\mathbf{x}_{N},$
but not on the extension $\mathbf{x.}.$
\end{theorem}

\begin{proof}
From the LV Extension Theorem, there always exists an extension $\mathbf{x}$
of $\left(  \mathbf{x}_{1},\mathbf{x}_{2},...,\mathbf{x}_{N}\right)  $ in
$WG\Omega_{q}\left(
\mathbb{R}
^{2d}\right)  $ for any $q>p$ $WG\Omega_{q}\left(
\mathbb{R}
^{Nd}\right)  $ (and any $q\geq p$, if $p$ is not an integer). Let us define
an algebra homomorphism from the (truncated) tensor algebra $T^{\lfloor
p\rfloor}\left(
\mathbb{R}
^{Nd}\right)  $ into the space of continuous differential operators, by taking
the linear extension of%
\[
F^{W}\left(  e_{i_{1}...i_{n}}\right)  =W^{i_{1}}\circ...\circ W^{i_{n}}.
\]
Restricting $F^{W}$ to $\mathfrak{g}^{\lfloor p\rfloor}\left(
\mathbb{R}
^{Nd}\right)  $ gives a Lie algebra homomorphism into the space of vector
fields on $%
\mathbb{R}
^{Nd}.$ An easy calculation confirms that
\begin{equation}
\ker\left(  F^{W}|_{\mathfrak{g}^{\lfloor p\rfloor}\left(
\mathbb{R}
^{Nd}\right)  }\right)  \supseteq\mathfrak{k}^{\lfloor p\rfloor}\left(
\mathbb{R}
^{Nd}\right)  ,
\end{equation}
whereupon Theorem 20 of \cite{LV} shows that $\mathbf{y}^{1}$ is independent
of the extension of $\left(  \mathbf{x}_{1},\mathbf{x}_{2},...,\mathbf{x}%
_{N}\right)  $ to $\mathbf{x.}$ In general, $\mathbf{y}^{2},...,\mathbf{y}%
^{\lfloor p\rfloor}$ will still depend on the choice of lift. Nevertheless,
the \textit{projections }of $\mathbf{y}$ to the $N$ paths $\mathbf{y}%
_{1},\mathbf{y}_{2},...,\mathbf{y}_{N}$ will be not do so. This is most easily
seen when $p\in\left(  2,3\right)  $ by the same calculation as in (\ref{sec}).
\end{proof}

\begin{remark}
\label{geom}Because each $\mathbf{y}_{i}=\Theta_{V^{0},V}\left(  \mu
,y_{i}\left(  0\right)  ,\mathbf{x}_{i}\right)  ,$ $\mathbf{y}_{i}$ solves an
RDE driven by $\mathbf{x}_{i}\in G\Omega_{p}\left(
\mathbb{R}
^{d}\right)  $ and the Universal Limit Theorem guarantees that $\mathbf{y}%
_{i}$ is in fact an element of $G\Omega_{p}\left(
\mathbb{R}
^{e}\right)  .$ This observation will be useful later on. It follows from
Theorem \ref{Npart}, together with a suitable elaboration of the arguments of
Corollary \ref{wd}, that
\[
\mu=u_{0}^{\otimes N}\ast\sum_{i=1}^{N}\lambda_{i}\delta_{\mathbf{y}_{i}}%
\]
is the unique fixed point of $\Psi_{\nu}.$
\end{remark}

\section{A fixed-point and continuity theorem\label{fps}}

We now want to consider the case where the preference measure $\nu$ is a
non-discrete measure on rough path space. The main problem we address is to
find a condition on $\nu$ to force the existence of a unique fixed point to
the map $\Psi_{\nu}.$ A key feature will be the use of estimates controlling:
\[
\rho_{p-var;\left[  0,T\right]  }\left(  \mathbf{y}^{1},\mathbf{y}^{2}\right)
,
\]
the $\rho_{p-var;\left[  0,T\right]  }$-distance between two RDE solutions
$\mathbf{y}^{1}$ and $\mathbf{y}^{2}$ driven by $\mathbf{x.}$ These estimates
need have two properties: they need to be Lipschitz in the defining data
(starting point, vector fields etc) \textit{and} the Lipschitz constant must
have integrable dependence on $\mathbf{x}$, when $\mathbf{x}$ is realised
according to a wide class of measures. Classical RDE estimates satisfying the
first of these criteria, the latter needs more work. For example in \cite{FV}
\ the authors have proved estimates of the form
\begin{equation}
\rho_{p-var;\left[  0,T\right]  }\left(  \mathbf{y}^{1},\mathbf{y}^{2}\right)
\leq\left(  \ast\right)  \exp\left(  \left\vert \left\vert \mathbf{x}%
\right\vert \right\vert _{p-var;\left[  0,T\right]  }^{p}\right)  ,
\label{fv est}%
\end{equation}
where the terms $\left(  \ast\right)  $ incorporate the data. The drawback of
this estimates is that the right hand side fails to be integrable, for example
when $\mathbf{x}$ is the lift of a wide class of common process including
Brownian motion and fractional Brownian motion with $H<1/2.$ Fortunately, it
it possible to replace $\left\vert \left\vert \mathbf{x}\right\vert
\right\vert _{p-var;\left[  0,T\right]  }^{p}$ in (\ref{fv est}) by a quantity
called the \textit{accumulated }$\alpha-$\textit{local }$p$\textit{-variation}
(see below). By then making use of the recent tail estimates in \cite{CLL}, we
are able to cover these interesting examples.

\subsection{Lipschitz-continuity for RDEs with drift}

We recall the definition of the following function from \cite{CLL}:

\begin{definition}
Let $\alpha>0$ and $I\subseteq%
\mathbb{R}
$ be a compact interval. Suppose that $\omega:I\times I\rightarrow%
\mathbb{R}
^{+}$ is a control. We define the accumulated $\alpha-$local $\omega
$-variation by%
\[
M_{\alpha,I}\left(  \omega\right)  =\sup_{\substack{D\left(  I\right)
=\left(  t_{i}\right)  \\\omega\left(  t_{i},t_{i+1}\right)  \leq\alpha}%
}\sum_{i:t_{i}\in D\left(  I\right)  }\omega\left(  t_{i},t_{i+1}\right)  .
\]

\end{definition}

\bigskip The following lemma is a Lipschitz estimate on the RDE solution (with
drift), when we vary the defining data of the differential equation.

\begin{lemma}
\label{lip est}Let $\gamma>p\geq1$ and $\beta>1.$ Suppose $\mathbf{x}$ is a
weakly geometric $p-$rough path in $WG\Omega_{p}\left(
\mathbb{R}
^{d}\right)  $, and assume that $\gamma^{1}$ and $\gamma^{2}$ are two paths
which take values in some Banach space $E$, and belong to $C^{1-var}\left(
[0,T],E\right)  .$ Then $\omega:\Delta_{\left[  0,T\right]  }\rightarrow%
\mathbb{R}
^{+}$ defined by
\[
\omega\left(  s,t\right)  :=\sum_{i=1}^{2}\left\vert \left\vert \gamma
^{i}\right\vert \right\vert _{1-var;\left[  s,t\right]  }+\sum_{i=1}%
^{2}\left\vert \left\vert \mathbf{x}^{i}\right\vert \right\vert
_{p-var;\left[  s,t\right]  }^{p}%
\]
is a control. Furthermore, if $V=\left(  V^{1},...,V^{d}\right)  $ is a
collection of vector fields in $Lip^{\gamma}\left(
\mathbb{R}
^{e}\right)  ,$ and $V^{0}$ is in $Lip^{\beta}\left(
\mathbb{R}
^{e},L\left(  E,%
\mathbb{R}
^{e}\right)  \right)  ,$ then for $i=1,2$ the RDEs%
\begin{align*}
d\mathbf{y}_{t}^{i}  &  =V\left(  y_{t}^{i}\right)  d\mathbf{x}_{t}^{i}%
+V^{0}\left(  y_{t}^{i}\right)  d\gamma_{t}^{i},\\
\pi_{1}\mathbf{y}_{t}^{i}  &  =y_{0}^{i}%
\end{align*}
have unique solutions. And for every $\alpha>0$ and some $C=C\left(
v,\alpha\right)  >0,$ we also have the following Lipschitz-continuity of the
solutions:%
\[
\rho_{p,\omega}\left(  \mathbf{y}^{1},\mathbf{y}^{2}\right)  \leq C\left[
\left\vert y_{0}^{1}-y_{0}^{2}\right\vert +\rho_{1,\omega}\left(  \gamma
^{1},\gamma^{2}\right)  +\rho_{p,\omega}\left(  \mathbf{x}^{1},\mathbf{x}%
^{2}\right)  \right]  \exp\left(  CM_{\alpha,\left[  0,T\right]  }\left(
\omega\right)  \right)  .
\]

\end{lemma}

\begin{proof}
The proof is obtained by following the arguments of Theorem 12.10 of \cite{FV}
on RDEs with drift; two enhancements are necessary. The first is allow the
drift term to take values in an arbitrary (infinite dimensional) Banach space.
This is elementary, because in the current lemma $\gamma^{1}$ and $\gamma^{2}$
have bounded variation, and hence classical ode estimates can be used
everywhere. The second, more subtle, enhancement is to end up with the
accumulated $\alpha-$local $\omega$-variation featuring in the exponential on
the right hand side (as opposed to the usual $\omega\left(  0,T\right)
$).\ For this we refer to \cite{CLL} and Remark 10.64 of \cite{FV}.
\end{proof}

By exploiting the relationship between $\rho_{p,\omega}$ and $\rho_{p-var} $
we can obtain a Lipschitz estimate in $\rho_{p-var}$-distance:

\begin{corollary}
\label{used bound}With the notation of, and under the same conditions as,
Lemma \ref{lip est}, we have
\begin{align*}
&  \rho_{p-var;\left[  0,T\right]  }\left(  \mathbf{y}^{1},\mathbf{y}%
^{2}\right) \\
&  \leq C\omega\left(  0,T\right)  ^{N}\left[  \left\vert y_{0}^{1}-y_{0}%
^{2}\right\vert +\left\vert \left\vert \gamma^{1}-\gamma^{2}\right\vert
\right\vert _{1-var;\left[  0,T\right]  }+\left\vert \left\vert \mathbf{x}%
^{1}-\mathbf{x}^{2}\right\vert \right\vert _{p-var;\left[  s,t\right]
}\right]  \exp\left(  CM_{\alpha,\left[  0,T\right]  }\left(  \omega\right)
\right)  ,
\end{align*}
for some $N>0$.
\end{corollary}

\subsection{Measure-valued paths}

For the current application, the main interest in these Lipschitz estimates
will occur when the space of probability measures $\mathcal{M}\left(
S_{p,e}\right)  $ is embedded in a Banach space $E$. In the typically case
$\gamma$ will then be constructed from $\mu\in\mathcal{M}\left(
P_{p,e}\right)  $ by setting $\gamma_{t}:=\int_{0}^{t}\mu_{s}ds.$ For the
moment, we develop this more abstractly by letting $Lip^{1}\left(  S\right)
^{\ast}$denote the dual of Lip-1 functions (that is, the bounded Lipschitz
functions) on a metric space $\left(  S,d\right)  .$ There is a canonical
injection $\mu\mapsto T_{\mu}$ from $\mathcal{M}\left(  S\right)  $ into
$Lip^{1}\left(  S\right)  ^{\ast}$ defined by the integration of functions in
$BL\left(  S\right)  $ against $\mu:$%
\begin{equation}
T_{\mu}\left(  \phi\right)  =\left\langle T_{\mu},\phi\right\rangle :=\int%
_{S}\phi\left(  s\right)  \mu\left(  ds\right)  . \label{embedd}%
\end{equation}
In this setting, two metric spaces will be of special interest as already
mentioned in Section \ref{rp}. The first is the step-$N$ free nilpotent group
with $e$ generators, $G^{N}\left(
\mathbb{R}
^{e}\right)  $, with the (inhomogeneous) metric it inherits from the tensor
algebra:%
\[
d_{N}\left(  \mathbf{g,h}\right)  :=\max_{i=1,...,N}\left\vert \pi_{i}\left(
\mathbf{g-h}\right)  \right\vert .
\]
The second is the space of geometric $p-$rough paths $G\Omega_{p}\left(
\mathbb{R}
^{d}\right)  $ equipped with $\rho_{p-var;\left[  0,T\right]  }.$

The next lemma examines the regularity of the paths which result from
(\ref{marginal}) the pushforward of $\mu$ under the evaluation maps, i.e.
\begin{equation}
\mu_{t}\equiv\left(  \psi_{t}\right)  _{\ast}\mu\in\mathcal{M}\left(
S_{p,e}\right)  . \label{marginal}%
\end{equation}

\begin{lemma}
Suppose $p\geq1$ and let $\mu$ be an element of $\mathcal{M}_{1}\left(
P_{p,e}\right)  .$ Let $\left(  \mu_{t}\right)  _{t\in\left[  0,T\right]  }$
be the path in $\mathcal{M}\left(  S_{p,e}\right)  $ defined by
(\ref{marginal}), and $\left(  T_{\mu_{t}}\right)  _{t\in\left[  0,T\right]
}$ the path in $Lip^{1}\left(  S_{p,e}\right)  ^{\ast}$ obtained by the
injection of $\mathcal{M}\left(  S_{p,e}\right)  $ into $Lip^{1}\left(
S_{p,e}\right)  ^{\ast}.$ Then for every $0\leq s\leq t\leq T$ we have%
\[
\left\vert \left\vert T_{\mu_{t}}-T_{\mu_{s}}\right\vert \right\vert
_{Lip^{1}\left(  S_{p,e}\right)  ^{\ast}}\leq\int_{P_{p,e}}d_{\lfloor
p\rfloor}\left(  \mathbf{y}_{s},\mathbf{y}_{t}\right)  \mu\left(
d\mathbf{y}\right)  \leq\int_{P_{p,e}}\rho_{p-var;\left[  s,t\right]  }\left(
1,\mathbf{y}\right)  \mu\left(  d\mathbf{y}\right)  ,
\]
where $\mathbf{y}_{u}=\psi_{u}\left(  \mathbf{y}\right)  .$ In particular,
$t\mapsto T_{\mu_{t}}$ is a continuous path in $Lip^{1}\left(  S_{p,e}\right)
^{\ast}.$
\end{lemma}

\begin{proof}
Take $\phi\in Lip^{1}\left(  S_{p}\right)  $ with $\left\vert \left\vert
\phi\right\vert \right\vert _{Lip^{1}\left(  S_{p,e}\right)  }=1.$ The result
then follows from the proceeding calculation:%
\begin{align*}
\left\vert \left\langle T_{\mu_{t}}-T_{\mu_{s}},\phi\right\rangle \right\vert
&  =\left\vert \int_{S_{p,e}}\phi\left(  \mathbf{g}\right)  \left[  \left(
\psi_{t}\right)  _{\ast}\mu\right]  \left(  d\mathbf{g}\right)  -\int%
_{S_{p,e}}\phi\left(  \mathbf{g}\right)  \left[  \left(  \psi_{s}\right)
_{\ast}\mu\right]  \left(  d\mathbf{g}\right)  \right\vert \\
&  =\left\vert \int_{P_{p,e}}\left(  \phi\circ\psi_{t}-\phi\circ\psi
_{s}\right)  \left(  \mathbf{y}\right)  \mu\left(  d\mathbf{y}\right)
\right\vert \\
&  \leq\int_{P_{p,e}}d_{\lfloor p\rfloor}\left(  \psi_{s}\left(
\mathbf{y}\right)  ,\psi_{t}\left(  \mathbf{y}\right)  \right)  \mu\left(
d\mathbf{y}\right) \\
&  \leq\int_{P_{p,e}}\rho_{p-var;\left[  s,t\right]  }\left(  1,\mathbf{y}%
\right)  \mu\left(  d\mathbf{y}\right)
\end{align*}
The right hand side is finite since $\mu$ is in $\mathcal{M}_{1}\left(
P_{p,e}\right)  ,$ and (by the dominated convergence theorem) it tends to zero
as $\left\vert t-s\right\vert $ tends to zero.
\end{proof}

It follows from this lemma that
\begin{equation}
\gamma_{s,t}:=\int_{s}^{t}T_{\mu_{r}}dr:=\int_{0}^{T}1_{\left[  s,t\right]
}\left(  r\right)  T_{\mu_{r}}dr\in Lip^{1}\left(  S_{p,e}\right)  ^{\ast}
\label{occ measure}%
\end{equation}
is well-defined for every $\left(  s,t\right)  \subseteq\left[  0,T\right]  ,$
where the integral is understood in the sense of Bochner integration. For any
$\phi\in Lip^{1}\left(  S_{p,e}\right)  ,$ standard properties of the integral
yield that%
\begin{equation}
\left\langle \gamma_{s,t},\phi\right\rangle =\int_{s}^{t}\left\langle
T_{\mu_{r}},\phi\right\rangle dr=\int_{s}^{t}\int_{S_{p,e}}\phi\left(
\mathbf{g}\right)  \mu_{r}\left(  d\mathbf{g}\right)  dr. \label{boch}%
\end{equation}
The space $P_{p,e}$ carries with it an implicit time interval $\left[
0,T\right]  $ which we suppress in the notation$.$ Occasionally, we might want
to make this explicit by writing $P_{p,e,T}$. For example, if we start with a
probability measure $\mu$ in $\mathcal{M}\left(  P_{p,e,T}\right)  $ we will
need to consider its restriction, $\mu|_{\left[  0,t\right]  }$ to a
probability measure in $\mathcal{M}\left(  P_{p,e,t}\right)  $. We then let
$W_{t}$ denote the Wasserstein metric on $\mathcal{M}\left(  P_{p,e,t}\right)
$, and write $W_{t}\left(  \mu^{1},\mu^{2}\right)  $ to mean $W_{t}\left(
\mu^{1}|_{\left[  0,t\right]  },\mu^{2}|_{\left[  0,t\right]  }\right)  $.

\begin{corollary}
\label{fv bound}Suppose $p\geq1$ and let $T>0$. Assume $\mu^{1}$ and $\mu^{2}$
are two elements of $\mathcal{M}_{1}\left(  P_{p,e,T}\right)  ,$ and for
$i=1,2$ let $\gamma^{i}:\Delta_{\left[  0,T\right]  }\rightarrow
Lip^{1}\left(  S_{p,e}\right)  ^{\ast}$ be the function defined by%
\[
\gamma_{s,t}^{i}=\int_{s}^{t}T_{\mu_{r}^{i}}dr.
\]
Then for every $0\leq s<t\leq T$ we have that
\begin{equation}
\left\vert \gamma_{s,t}^{1}-\gamma_{s,t}^{2}\right\vert _{Lip^{1}\left(
S_{p,e}\right)  ^{\ast}}\leq C\left(  t-s\right)  W_{t}\left(  \mu^{1},\mu
^{2}\right)  . \label{lip}%
\end{equation}
In particular,
\begin{equation}
\left\vert \left\vert \gamma^{1}-\gamma^{2}\right\vert \right\vert
_{1-var;\left[  0,T\right]  }\leq C\int_{0}^{T}W_{t}\left(  \mu^{1},\mu
^{2}\right)  dt. \label{var}%
\end{equation}

\end{corollary}

\begin{proof}
Let $\phi\in Lip^{1}\left(  S_{p,e}\right)  $ with $\left\vert \left\vert
\phi\right\vert \right\vert _{Lip^{1}\left(  S_{p}\right)  ^{\ast}}=1,$ then
from (\ref{boch}) we can deduce that
\begin{align*}
\left\vert \left\langle \gamma_{s,t}^{1}-\gamma_{s,t}^{2},\phi\right\rangle
\right\vert  &  =\left\vert \int_{s}^{t}\left[  \int_{S_{p,e}}\phi\left(
\mathbf{g}\right)  \mu_{r}^{1}\left(  d\mathbf{g}\right)  -\int_{S_{p}}%
\phi\left(  \mathbf{g}\right)  \mu_{r}^{2}\left(  d\mathbf{g}\right)  \right]
dr\right\vert \\
&  \leq\left(  t-s\right)  \sup_{r\in\left[  s,t\right]  }\int_{S_{p,e}\times
S_{p,e}}d_{\lfloor p\rfloor}\left(  \mathbf{g}^{1},\mathbf{g}^{2}\right)
\pi_{r}\left(  d\mathbf{g}^{1},d\mathbf{g}^{2}\right)  ,
\end{align*}
where $\pi_{r}$ is an element of $\mathcal{M}_{1}\left(  S_{p,e}\times
S_{p,e}\right)  $ with marginal distributions $\mu_{r}^{1}$ and $\mu_{r}^{2}$.
$\ $For any such $\pi_{r}$ and every $r\in\left[  s,t\right]  $ we have%
\[
\int_{S_{p,e}\times S_{p,e}}d_{\lfloor p\rfloor}\left(  \mathbf{g}%
^{1},\mathbf{g}^{2}\right)  \pi_{r}\left(  d\mathbf{g}^{1},d\mathbf{g}%
^{2}\right)  \leq C\int_{P_{p,e,t}\times P_{p,e,t}}\rho_{p-var;\left[
0,t\right]  }\left(  \mathbf{y}^{1},\mathbf{y}^{2}\right)  \pi\left(
d\mathbf{y}^{1},d\mathbf{y}^{2}\right)  ,
\]
where $\pi$ in $\mathcal{M}_{1}\left(  P_{p,e,t}\times P_{p,e,t}\right)  $ is
any coupling of $\mu^{1}|_{\left[  0,t\right]  }$ and $\mu^{2}|_{\left[
0,t\right]  }.$ This yields
\[
\sup_{r\in\left[  s,t\right]  }\int_{S_{p,e}\times S_{p,e}}d_{\lfloor
p\rfloor}\left(  \mathbf{g}^{1},\mathbf{g}^{2}\right)  \pi_{r}\left(
d\mathbf{g}^{1},d\mathbf{g}^{2}\right)  \leq W_{t}\left(  \mu^{1},\mu
^{2}\right)  ,
\]
which implies (\ref{lip}) at once. Deducing (\ref{var}) from (\ref{lip}) is
then elementary.
\end{proof}

\subsection{A fixed-point theorem}

Suppose that $p\geq1$ and $\mathbf{x}$ is an element of $G\Omega_{p}\left(
\mathbb{R}
^{d}\right)  ,$ then we write $\omega_{\mathbf{x}}$ for the control induced by
$\mathbf{x}$ via
\[
\omega_{\mathbf{x}}\left(  s,t\right)  \equiv\left\vert \left\vert
\mathbf{x}\right\vert \right\vert _{p-var;\left[  s,t\right]  }^{p}.
\]
The following lemma gives a useful way of controlling $\omega_{\mathbf{x}%
}\left(  0,T\right)  $ in terms of the $\alpha$-local $p$-variation.

\begin{lemma}
\label{p-var bound}For any $\mathbf{x}$ in $G\Omega_{p}\left(
\mathbb{R}
^{d}\right)  $ and any $\alpha>0,$ we have that
\[
\left\vert \left\vert \mathbf{x}\right\vert \right\vert _{p-var;\left[
0,T\right]  }^{p}=\omega_{\mathbf{x}}\left(  0,T\right)  \leq2^{p-1}\alpha
\max\left\{  1,\alpha^{-p}M_{\alpha,\left[  0,T\right]  }\left(
\omega_{\mathbf{x}}\right)  ^{p}\right\}
\]

\end{lemma}

\begin{proof}
Fix $\alpha>0$, and let $D=\left(  t_{i}:i=0,1....,n\right)  $ be an arbitrary
partition of $\left[  0,T\right]  .$ We aim to estimate%
\[
\sum_{i=1}^{n}\left\vert \left\vert \mathbf{x}_{t_{i-1},t_{i}}\right\vert
\right\vert ^{p}:=\sum_{i=1}^{n}d_{CC}\left(  \mathbf{x}_{t_{i-1}}%
,\mathbf{x}_{t_{i}}\right)  ^{p}%
\]
Let $t_{j-1}$ and $t_{j}$ be any two consecutive points in $D,$ and define
$\sigma_{0}=t_{j-1}$ and
\[
\sigma_{i+1}=\inf\left\{  t>\sigma_{i}:\omega\left(  \sigma_{i},t\right)
=\alpha\right\}  \wedge t_{j}%
\]
for $i\in%
\mathbb{N}
$. \ Define
\[
N_{\alpha,\left[  t_{j-1},t_{j}\right]  }\left(  \omega_{\mathbf{x}}\right)
=\sup\left\{  n\in%
\mathbb{N}
\cup\left\{  0\right\}  :\sigma_{n}<t_{j}\right\}
\]
\ A simple calculation shows that $t_{N}=t_{j}$ if $N=N_{\alpha,\left[
t_{j-1},t_{j}\right]  }\left(  \omega_{\mathbf{x}}\right)  +1$, and therefore%
\begin{align*}
\left\vert \left\vert \mathbf{x}_{t_{j-1},t_{j}}\right\vert \right\vert ^{p}
&  \leq\left(  \sum_{i=1}^{N+1}\left\vert \left\vert \mathbf{x}_{\sigma
_{i-1},\sigma_{j}}\right\vert \right\vert \right)  ^{p}\\
&  \leq\left(  N+1\right)  ^{p-1}\sum_{i=1}^{N+1}\left\vert \left\vert
\mathbf{x}_{\sigma_{i-1},\sigma_{j}}\right\vert \right\vert ^{p}\leq\left(
N_{\alpha,\left[  0,T\right]  }\left(  \omega_{\mathbf{x}}\right)  +1\right)
^{p-1}\sum_{i=1}^{N+1}\left\vert \left\vert \mathbf{x}_{\sigma_{i-1}%
,\sigma_{j}}\right\vert \right\vert ^{p}.
\end{align*}
Using this observation it is easy to deduce that
\begin{equation}
\sum_{i=1}^{n}\left\vert \left\vert \mathbf{x}_{t_{i-1},t_{i}}\right\vert
\right\vert ^{p}\leq\left(  N_{\alpha,\left[  0,T\right]  }\left(
\omega_{\mathbf{x}}\right)  +1\right)  ^{p-1}M_{\alpha,\left[  0,T\right]
}\left(  \omega_{\mathbf{x}}\right)  . \label{unif bound}%
\end{equation}
The claimed bounded follows by first noticing that $N_{\alpha,\left[
0,T\right]  }\left(  \omega_{\mathbf{x}}\right)  \leq\alpha^{-1}%
M_{\alpha,\left[  0,T\right]  }\left(  \omega\right)  ,$ and then taking the
supremum over all partitions $D$ in (\ref{unif bound}).
\end{proof}

In order to prove the fixed point theorem we require integrability on the
preference measure. The subset of $\mathcal{M}\left(  P_{p,d}\right)  $ for
which the fixed-point theorem will hold is described by the following condition.

\begin{condition}
\label{exp inte}Let $p\geq1$. $\nu$ will denote a probability measure in
$\mathcal{M}\left(  P_{p,d}\right)  ,$ and $\phi_{\nu}$ will be pushforward
measure in $\mathcal{M}\left(  [0,\infty)\right)  $ defined by
\[
\phi_{\nu}:=\left[  M_{1,\left[  0,T\right]  }\left(  \omega_{\mathbf{\cdot}%
}\right)  \right]  _{\ast}\left(  \nu\right)  .
\]
We will assume that $\phi_{\nu}$ has well-defined moment-generating function;
i.e. for every $\theta$ in $%
\mathbb{R}
$ we have%
\[
\int_{\lbrack0,\infty)}\exp\left[  \theta y\right]  \phi_{\nu}\left(
dy\right)  =\int_{P_{p,d}}\exp\left[  \theta M_{1,\left[  0,T\right]  }\left(
\omega_{\mathbf{x}}\right)  \right]  \nu\left(  d\mathbf{x}\right)  <\infty.
\]

\end{condition}

\begin{remark}
If Condition \ref{exp inte} is in force, then $\phi_{\nu}^{\alpha}:=\left[
M_{\alpha,\left[  0,T\right]  }\left(  \omega_{\mathbf{\cdot}}\right)
\right]  _{\ast}\left(  \nu\right)  $ will also have a well-defined moment
generating function for any $\alpha$ in $\left(  0,1\right)  $.
\end{remark}

For the reader's convenience, we recall some notation from Section
\ref{finite}. $\Psi:\mathcal{M}_{1}\left(  P_{p,e}\right)  \rightarrow
\mathcal{M}_{1}\left(  P_{p,e}\right)  $ is defined by
\[
\Psi=\Psi_{\nu}:\mu\mapsto\left[  \Theta_{V^{0},V}\left(  \mu,\cdot
,\cdot\right)  \right]  _{\ast}\left(  u_{0}\times\nu\right)  \in
\mathcal{M}_{1}\left(  P_{p,e}\right)  ,
\]
and fixed points of $\Psi_{\nu}$ correspond to solutions of the nonlinear
McKean-Vlasov RDE%
\begin{equation}
\left\{
\begin{array}
[c]{c}%
d\mathbf{Y}_{t}=V\left(  Y_{t}^{\mu}\right)  d\mathbf{X}_{t}+V^{0}\left(
Y_{t}^{\mu}\right)  d\gamma_{t}^{\mu}\\
\text{Law}\left(  \mathbf{Y}\right)  =\mu,\text{ Law}\left(  Y_{0}\right)
=u_{0}%
\end{array}
\right.  . \label{mkvn2}%
\end{equation}
We now formulate and prove our main existence and uniqueness theorem for
solutions to (\ref{mkvn2}).

\begin{theorem}
\label{fixed point}Let $\gamma>p\geq1$ and $\beta>1.$ Suppose $\nu$ be an
element of $\mathcal{M}_{1}\left(  P_{p,d}\right)  $ which satisfies Condition
\ref{exp inte}. Let $\left(  Y_{0},\mathbf{X}\right)  $ be a random variable
on a probability space $\left(  \Omega,\mathcal{F},P\right)  ,$ taking values
in $%
\mathbb{R}
^{e}\times P_{p,d}$ and having law $u_{0}\times\nu.$ Then for any collection
of vector fields $V=\left(  V^{1},...,V^{d}\right)  $ in $Lip^{\gamma}\left(
\mathbb{R}
^{e}\right)  $ and $V^{0}$ in $Lip^{\beta}\left(
\mathbb{R}
^{e},L\left(  Lip^{1}\left(  S_{p,e}\right)  ^{\ast},%
\mathbb{R}
^{e}\right)  \right)  $, there exists a unique solution to the nonlinear
McKean-Vlasov RDE (\ref{mkvn2}).
\end{theorem}

\begin{proof}
The space $\left(  P_{p,e},\rho_{p}\right)  $ \ is complete, and hence (see,
e.g., \cite{vill}) so is $\left(  \mathcal{M}_{1}\left(  P_{p,e}\right)
,W\right)  .$ Suppose $\mu_{1}$ and $\mu_{2}$ are in \ $\mathcal{M}_{1}\left(
P_{p,e}\right)  ,$ and let
\[
\mathbf{Y}^{i}=\Theta_{V^{0},V}\left(  \mu_{i},Y_{0},\mathbf{X}\right)  .
\]
Using Corollary \ref{used bound} together with Lemma \ref{p-var bound} we
obtain for any $\alpha$ in $(0,1]$ the bound%
\[
\rho_{p-var;\left[  0,T\right]  }\left(  \mathbf{Y}^{1},\mathbf{Y}^{2}\right)
\leq C\left\vert \left\vert \gamma^{\mu_{1}}-\gamma^{\mu_{2}}\right\vert
\right\vert _{1-var;\left[  0,T\right]  }\exp\left(  CM_{\alpha,\left[
0,T\right]  }\left(  \omega_{\mathbf{X}}\right)  \right)  ,
\]
for some $C=C_{1}\left(  \alpha\right)  >0.$ It follows from Corollary
\ref{fv bound} that
\[
\left\vert \left\vert \gamma^{\mu_{1}}-\gamma^{\mu_{2}}\right\vert \right\vert
_{1-var;\left[  0,T\right]  }\lesssim\int_{0}^{T}W_{t}\left(  \mu_{1},\mu
_{2}\right)  dt.
\]
And therefore by taking expectations in the previous inequality we obtain
\begin{align*}
W_{T}\left(  \Psi_{\nu}\left(  \mu_{1}\right)  ,\Psi_{\nu}\left(  \mu
_{2}\right)  \right)   &  \leq E\left[  \rho_{p-var;\left[  0,T\right]
}\left(  \mathbf{Y}^{1},\mathbf{Y}^{2}\right)  \right] \\
&  \leq CE\left[  \exp\left(  CM_{\alpha,\left[  0,T\right]  }\left(
\omega_{\mathbf{X}}\right)  \right)  \right]  \int_{0}^{T}W_{t}\left(  \mu
_{1},\mu_{2}\right)  dt,
\end{align*}
where Condition \ref{exp inte} ensures that the right hand side is finite. It
from a standard Banach-type contraction argument that the map $\Psi_{\nu}$ has
a unique fixed point.
\end{proof}

In light of the conditions of this theorem, it is useful to make some
observation about the type of processes which satisfy the key integrability
condition (Condition \ref{exp inte}). In the recent paper \cite{CLL} we
consider a continuous Gaussian process $X=\left(  X^{1},..,X^{d}\right)  $
with i.i.d. components such that:

\begin{enumerate}
\item $X$ has a natural lift to a geometric $p$-rough path $\mathbf{X;}$

\item The Cameron-Martin space associated to $X$ has the embedding property%
\[
\mathcal{H}\hookrightarrow C^{q\text{-var}}\left(  \left[  0,T\right]  ,%
\mathbb{R}
^{d}\right)
\]
for some $1/p+1/q>1.$
\end{enumerate}

We then prove that for some $\eta>0$ we have
\[
E\left[  \exp\left[  \eta M_{\alpha,\left[  0,T\right]  }\left(
\omega_{\mathbf{X}}\right)  ^{2/q}\right]  \right]  <\infty.
\]
This class of examples is rich enough to include fractional Brownian motion
$H>1/4$ (for which $q$ can be chosen to ensure $2/q>1$)$,$ and other examples
of Gaussian processes which are genuinely rougher than Brownian motion (see
\cite{FV07}). The importance of the Lipschitz estimate in Corollary
\ref{used bound} can now be grasped more clearly. Since, as an immediate
corollary, we see that Condition \ref{exp inte} holds for the class of
measures described.

\begin{remark}
In some recent work \cite{ball},a flow-based approach is used to derive
continuity estimates for RDEs. An existence and uniqueness theorem is proved,
under the following condition on $\mathbf{X}$: for some family of random
variables $\left\{  C_{s}:s\in\left[  0,T\right]  \right\}  ,$ which is
bounded in $L^{1},$
\begin{equation}
\left\vert E\left[  \mathbf{X}_{s,t}^{k}|\mathcal{F}_{s}\right]  \right\vert
\leq C_{s}\left(  t-s\right)  ,\text{ }\forall\left[  s,t\right]
\subseteq\left[  0,T\right]  ,\text{ }k=1,...,\left\lfloor p\right\rfloor
.\label{SCALING}%
\end{equation}
This requirement forces some structure upon the sample paths of $X$, (for
example: smoothness, or independence of increments). It does not hold in
general for the examples illustrated above, where the sample paths are less
regular than Brownian motion. Indeed if $X$ is fBm with $H<1/2,$ then we have%
\[
t^{-1}E\left[  X_{0,t}^{2}\right]  =t^{2H-1}\uparrow\infty\text{ \ as
}t\downarrow0,
\]
which violates (\ref{SCALING}) when $s=0$. By contrast, the exponential
integrability required\ in Condition \ref{exp inte} holds both in this
example, and for the much wider class of Gaussian processes highlighted above.
\end{remark}

\subsection{Continuity in $\nu$}

Suppose we have a set of preference measures and for each measure in the set
the conditions of Theorem \ref{fixed point} hold, so that $\Psi_{\nu}\left(
\cdot\right)  $ has a unique fixed point. A very natural question is to ask
about the stability properties of this map. The rough path setup is
well-suited to tackle this sort of problem. To this end, let $K:\left(
0,\infty\right)  \longrightarrow\left(  0,\infty\right)  $ be a monotone
increasing real-valued function. Define a subset of $\mathcal{M}_{1}\left(
P_{p,d}\right)  $ by
\[
\mathcal{E}\left(  K\right)  =\left\{  \nu\in\mathcal{M}_{1}\left(
P_{p,d}\right)  :\forall\theta\in\left(  0,\infty\right)  ,\psi_{\nu}\left(
\theta\right)  \leq K\left(  \theta\right)  \right\}  ,
\]
where, as above,
\[
\psi_{\nu}\left(  \theta\right)  =\int_{[0,\infty)}\exp\left(  \theta
y\right)  \phi_{\nu}\left(  dy\right)  .
\]
It is easy to see that $\mathcal{E}\left(  K\right)  $ is a closed subset of
$\mathcal{M}_{1}\left(  P_{p,d}\right)  $ in the topology of weak convergence
of measures.

\begin{lemma}
Let $K:\left(  0,\infty\right)  \longrightarrow\left(  0,\infty\right)  $ be a
monotone increasing, then the map%
\begin{align*}
\Xi &  :\mathcal{E}\left(  K\right)  \rightarrow\mathcal{M}_{1}\left(
P_{e,d}\right) \\
\Xi &  :\nu\mapsto\text{ fixed point of }\Psi_{\nu}\left(  \cdot\right)
\end{align*}
is well-defined and continuous in the topology of weak convergence of measures
on\ $\mathcal{E}\left(  K\right)  .$
\end{lemma}

\begin{proof}
Suppose that $\left(  \nu_{n}\right)  _{n=1}^{\infty}$ is a sequence of
measures in $\mathcal{E}\left(  K\right)  $ such that $\nu_{n}\Rightarrow
\nu\in\mathcal{E}\left(  K\right)  $ as $n\rightarrow\infty$; we will show
that $W_{T}\left(  \Xi\left(  \nu_{n}\right)  ,\Xi\left(  \nu\right)  \right)
\rightarrow0$ as $n\rightarrow\infty.$ By Skorohod's lemma there exists a
probability space carrying: (i) an $%
\mathbb{R}
^{e}$-valued random variable $Y_{0}$ with law $u_{0}$, and (ii) a sequence of
($P_{p,d}$-valued) random variables $\left(  \mathbf{X}^{\nu_{n}}\right)
_{n=1}^{\infty}$ and \ $\mathbf{X}^{\nu}$, such that $\mathbf{X}^{\nu}$ has
law $\nu$, $\mathbf{X}^{\nu_{n}}$ has law $\nu_{n}$ for every $n,$ and
\[
\mathbf{X}^{\nu_{n}}\rightarrow\ \mathbf{X}^{\nu}\text{ a.s. in }%
\rho_{p\text{-var}}.
\]
Let $\Xi\left(  \nu_{n}\right)  =\mu_{n}$, $\Xi\left(  \nu\right)  =\mu$ \ and
$\mathbf{Y}^{\nu_{n}}=\Theta_{V^{0},V}\left(  \mu_{n},Y_{0},\mathbf{X}%
^{\nu_{n}}\right)  $, $\mathbf{Y}^{\nu}=\Theta_{V^{0},V}\left(  \mu
,Y_{0},\mathbf{X}^{\nu}\right)  .$ We then have that
\[
\text{Law}\left(  \mathbf{Y}^{\nu_{n}}\right)  =\mu_{n}=\left[  \Theta
_{V_{0},V}\left(  \mu_{n},\cdot,\mathbf{\cdot}\right)  \right]  _{\ast}\left(
u_{0},\nu_{n}\right)  ,
\]
and similarly for Law$\left(  \mathbf{Y}^{\nu}\right)  .$ The estimates of
Corollary \ref{used bound} and Lemma \ref{p-var bound} show that for some
non-random $C_{1}>0:$
\[
\rho\left(  \mathbf{Y}^{\mathbf{\nu}_{n}},\mathbf{Y}^{\mathbf{\nu}}\right)
\leqslant C_{1}\left[  \left\vert \left\vert \gamma^{\nu_{n}}-\gamma^{\nu
}\right\vert \right\vert _{1-var;\left[  0,T\right]  }+\left\vert \left\vert
\mathbf{X}^{\nu_{n}}-\mathbf{X}^{\nu}\right\vert \right\vert _{p-var;\left[
0,T\right]  }\right]  \exp\left(  C_{1}M_{\alpha,\left[  0,T\right]  }\left(
\omega^{\nu_{n},\nu}\right)  \right)  ,
\]
where
\[
\omega^{\nu_{n},\nu}\left(  s,t\right)  \equiv\left\vert \left\vert
\mathbf{X}^{\nu_{n}}\right\vert \right\vert _{p-var;\left[  s,t\right]  }%
^{p}+\left\vert \left\vert \mathbf{X}^{\nu}\right\vert \right\vert
_{p-var;\left[  s,t\right]  }^{p}.
\]
Taking expectations and then making use of Corollary \ref{fv bound} and
Condition \ref{exp inte} it is easy to derive that
\begin{equation}
W_{T}\left(  \mu_{n},\mu\right)  \leq C_{2}\int_{0}^{T}W_{t}\left(  \mu
_{n},\mu\right)  dt+C_{1}a_{n}, \label{b}%
\end{equation}
where
\[
a_{n}:=E\left[  \left\vert \left\vert \mathbf{X}^{\nu_{n}}-\mathbf{X}^{\nu
}\right\vert \right\vert _{p-var;\left[  0,T\right]  }\exp\left(
CM_{\alpha,\left[  0,T\right]  }\left(  \omega^{\nu_{n},\nu}\right)  \right)
\right]  .
\]
A simple argument using the definition\footnote{In particular boundedness
(uniform in $n$) of the moment generating functions of $\nu_{n}$.} of
$\mathcal{E}\left(  K\right)  $ shows that $a_{n}\rightarrow0$ as
$n\rightarrow\infty.$ Finally, we use Gronwall's inequality in (\ref{b}) to
give%
\[
W_{T}\left(  \mu_{n},\mu\right)  \leq C_{3}a_{n}\exp\left(  C_{3}T\right)
\rightarrow0\text{ as }n\rightarrow\infty.
\]

\end{proof}

\section{Applications\label{propa}}

As an application of our uniqueness theorem, we prove the classical
\textit{propagation of chaos }phenomenon (see Sznitman \cite{sznitman}) for
the finite interacting particle system. This is the observation that, granted
sufficient symmetry to the interaction and initial configuration, then in the
large-population limit any finite subcollection of particles resembles the
evolution independent particles, each having the law of the nonlinear
McKean-Vlasov RDE. To make progress, let $\nu$ be a fixed preference measure
in $\mathcal{M}_{1}\left(  P_{p,d}\right)  $ which satisfies Condition
\ref{exp inte}. Assume that $\left\{  \left(  \mathbf{X}^{i},Y_{0}^{i}\right)
:i\in%
\mathbb{N}
\right\}  $ and $\left(  \mathbf{X,}Y_{0}\right)  $ are i.i.d. $%
\mathbb{R}
^{e}\times P_{p,d}-$valued random variables each with law $u_{0}\times\nu$ and
defined on the same probability space $\left(  \Omega,\mathcal{F},P\right)  $.
Suppose $\mu$ is the unique fixed point of $\Psi_{\nu}\left(  \cdot\right)  ,$
and let
\[
\mathbf{Y}=\Theta_{V^{0},V}\left(  \mu,Y_{0},\mathbf{X}\right)  .
\]
From Lemma \ref{wd} we can interpret the trajectories of individual particles
in the community%
\[
d\mathbf{Y}_{t}^{i,N}=\frac{1}{N}\sum_{j=1}^{N}\sigma\left(  Y_{t}^{i,N}%
,Y_{t}^{j,N}\right)  dt+V\left(  Y_{t}^{i}\right)  d\mathbf{X}_{t}^{i},\text{
}Y_{0}^{i,N}=Y_{0}^{i}%
\]
as projections of the solution to a system of rough differential equation
driven by any rough path in $P_{p,Nd}$ which consistently lifts $\mathbf{X}%
^{1},...,\mathbf{X}^{N}$. Indeed, we showed that $\mathbf{Y}^{i}$ is then
well-defined as%
\[
\mathbf{Y}^{i,N}=\Theta_{V^{0},V}\left(  \mu^{N},Y_{0}^{i},\mathbf{X}%
^{i}\right)  ,\text{ }i=1,....,N
\]
where $V^{0}$ in $Lip^{\beta}\left(
\mathbb{R}
^{e},L\left(  Lip^{1}\left(  S_{p,e}\right)  ^{\ast},%
\mathbb{R}
^{e}\right)  \right)  $ is defined in terms of the interaction kernel $\sigma$
by
\begin{equation}
V^{0}\left(  y\right)  \left(  \mu\right)  =\left\langle \mu,\sigma\left(
y,\cdot\right)  \right\rangle , \label{inter}%
\end{equation}
and $\mu^{N}=\mu^{N}\left(  \omega\right)  $ is the empirical measure
\[
\mu^{N}=\frac{1}{N}\sum_{j=1}^{N}\delta_{\mathbf{Y}^{j,N}}.
\]
For every $N,$ $\left\{  \mathbf{Y}^{1,N},\mathbf{Y}^{2,N},...,\mathbf{Y}%
^{N,N}\right\}  $ is an exchangeable system of random variables, and a
classical result of \cite{sznitman} shows that propagation of chaos is
equivalent to proving that
\[
\mu^{N}\Rightarrow\mu.
\]

\begin{remark}
Explicitly, this assertion says that the law of the random variable
\[
\mu^{N}\left(  \omega\right)  \in\mathcal{M}_{1}\left(  P_{p}\right)  ,
\]
which is a probability measure in $\mathcal{M}_{1}\left(  \mathcal{M}%
_{1}\left(  P_{p}\right)  \right)  ,$ converges weakly as $N\rightarrow\infty$
to the probability measure $\delta_{\mu},$ which is the law of the constant
random variable $\mu.$
\end{remark}

\begin{theorem}
Let $\gamma>p\geq1$, $\beta>1$ and $y_{0}\in%
\mathbb{R}
^{e}.$ Suppose that $\nu$ is a given preference measure in $\mathcal{M}%
_{1}\left(  P_{p,d}\right)  $ which satisfies Condition \ref{exp inte}. Assume
that the vector fields $V=\left(  V^{1},...,V^{d}\right)  $ belong to
$Lip^{\gamma}\left(
\mathbb{R}
^{e}\right)  ,$ and $V^{0}$ defined by (\ref{inter}) is in $Lip^{\beta}\left(
%
\mathbb{R}
^{e},L\left(  Lip\left(  S_{p,e}\right)  ^{\ast},%
\mathbb{R}
^{e}\right)  \right)  .$ Let $\mu$ denote the unique fixed point of
map$\ \Psi_{\nu}\left(  \cdot\right)  $ which results from\ Theorem
\ref{fixed point}. Assume further that $\left\{  \left(  Y_{0}^{i}%
,\mathbf{X}^{i}\right)  :i\in%
\mathbb{N}
\right\}  $ is a collection of i.i.d. $%
\mathbb{R}
^{e}\times P_{p,d}-$valued random variables, with law $u_{0}\times\nu,$
defined on the same probability space. For each $N\in%
\mathbb{N}
$ let $\left\{  \mathbf{Y}^{i,N}:1=1,....,N\right\}  $ be the solution to the
particle system
\[
d\mathbf{Y}_{t}^{i,N}=\frac{1}{N}\sum_{j=1}^{N}\sigma\left(  Y_{t}^{i,N}%
,Y_{t}^{j,N}\right)  dt+V\left(  Y_{t}^{i}\right)  d\mathbf{X}_{t}%
^{i},\mathbf{Y}_{0}^{i,N}=Y_{0}^{i}\in%
\mathbb{R}
^{e}.
\]
Then as $N\rightarrow\infty$
\begin{equation}
\frac{1}{N}\sum_{j=1}^{N}\delta_{\mathbf{Y}^{j,N}}=:\mu^{N}\Rightarrow\mu,
\label{prop}%
\end{equation}
and the particle system exhibits propagation of chaos.
\end{theorem}

\begin{proof}
\bigskip Fix $N\in%
\mathbb{N}
$ and suppose $i\in\left\{  1,...N\right\}  $. Let $\mathbf{Y}^{i}%
=\Theta_{V^{0},V}\left(  \mu,Y_{0}^{i},\mathbf{X}^{i}\right)  $ so that the
law $\mathbf{Y}^{i}$ is $\mu$, and $\left\{  \mathbf{Y}^{1},....,\mathbf{Y}%
^{N}\right\}  $ are $N$ independent copies of the solution of the rough
McKean-Vlasov equation. Then using Corollary \ref{used bound} and Lemma
\ref{p-var bound} we have (for $\alpha$ in $(0,1])$ that
\begin{equation}
\rho_{p-var;\left[  0,T\right]  }\left(  \mathbf{Y}^{i},\mathbf{Y}%
^{i,N}\right)  \leq C_{1}\left\vert \left\vert \gamma^{\mu}-\gamma^{\mu^{N}%
}\right\vert \right\vert _{1-var;\left[  0,T\right]  }\exp\left(
C_{1}M_{\alpha,\left[  0,T\right]  }\left(  \omega_{\mathbf{X}^{i}}\right)
\right)  . \label{exp}%
\end{equation}
Using Corollary \ref{fv bound} we observe
\[
\left\vert \left\vert \gamma^{\mu}-\gamma^{\mu^{N}}\right\vert \right\vert
_{1-var;\left[  0,T\right]  }\leq C_{2}\int_{0}^{T}W_{t}\left(  \mu,\mu
^{N}\right)  dt,
\]
and hence by summing (\ref{exp}) over $i=1,...,N$ we obtain
\begin{equation}
\frac{1}{N}\sum_{i=1}^{N}\rho_{p-var;\left[  0,T\right]  }\left(
\mathbf{Y}^{i},\mathbf{Y}^{i,N}\right)  \leq C_{2}\int_{0}^{T}W_{t}\left(
\mu,\mu^{N}\right)  dt\frac{1}{N}\sum_{i=1}^{N}\exp\left(  C_{1}%
M_{\alpha,\left[  0,T\right]  }\left(  \omega_{\mathbf{X}^{i}}\right)
\right)  . \label{b1}%
\end{equation}
Let%
\[
\bar{\mu}^{N}=\frac{1}{N}\sum_{j=1}^{N}\delta_{\mathbf{Y}^{j}}.
\]

Then we also have
\begin{equation}
W_{T}\left(  \mu,\mu^{N}\right)  \leq W_{T}\left(  \mu,\bar{\mu}^{N}\right)
+W_{T}\left(  \bar{\mu}^{N},\mu^{N}\right)  . \label{b2}%
\end{equation}
And, on the other hand using (\ref{b1}) we have the bound%
\begin{equation}
W_{T}\left(  \bar{\mu}^{N},\mu^{N}\right)  \leq C_{2}\int_{0}^{T}W_{t}\left(
\mu,\mu^{N}\right)  dt\frac{1}{N}\sum_{i=1}^{N}\exp\left(  C_{1}%
M_{\alpha,\left[  0,T\right]  }\left(  \omega_{\mathbf{X}^{i}}\right)
\right)  . \label{b3}%
\end{equation}
Putting (\ref{b3}) into (\ref{b2}) and using Gronwall's lemma we deduce that
\[
W_{T}\left(  \mu,\mu^{N}\right)  \leq C_{2}W_{T}\left(  \mu,\bar{\mu}%
^{N}\right)  \exp\left[  \frac{C_{1}}{N}\sum_{i=1}^{N}\exp\left(
C_{1}M_{\alpha,\left[  0,T\right]  }\left(  \omega_{\mathbf{X}^{i}}\right)
\right)  \right]  .
\]
It is a simple matter to conclude from the strong law of large numbers that
both $W_{T}\left(  \mu,\bar{\mu}^{N}\right)  \rightarrow0$ a.s., and
\[
\frac{1}{N}\sum_{i=1}^{N}\exp\left(  C_{1}M_{\alpha,\left[  0,T\right]
}\left(  \omega_{\mathbf{X}^{i}}\right)  \right)  \rightarrow E\left[
\exp\left(  C_{1}M_{\alpha,\left[  0,T\right]  }\left(  \omega_{\mathbf{X}%
}\right)  \right)  \right]  <\infty,
\]
a.s. as $N\rightarrow\infty.$ It is then easy to deduce (\ref{prop}).
Propagation of chaos is then a consequence of the classical result of
\cite{sznitman} we cited earlier.
\end{proof}

There are a number of follow-up results that seem worth pursuing. For example,
Sanov-type theorems \`{a} la Dawson-Gartner \cite{dawson} will be possible for
(\ref{prop}) in the weakly interacting case. Indeed, the presence of the rough
path topology, in which the universal limit theorem guarantees the continuity
of the It\^{o} map seems to simplify things greatly.\ We will return to these
discussion in future work.

\bigskip

\bigskip

\bigskip

\bigskip

\bigskip

\bigskip

\bigskip

\bigskip

\bigskip
\end{document}